\documentclass{amsart}
\usepackage{amssymb,amsmath,amsthm}
\usepackage[foot]{amsaddr}

\usepackage[backend=bibtex, style=numeric, giveninits=true, sortcites=true]{biblatex}
\usepackage{ast2}
\usepackage{bussproofs}
\usepackage[margin=1.7in]{geometry}
\newcommand{\op}{\ensuremath{^{\mathrm{op}}}}

\newcommand{\cat}[1]{\ensuremath{\mathbf{#1}}}
\usepackage{tikz-cd}

\usepackage{float}
\restylefloat{table}

\usepackage{quiver}
\usepackage{bm}
\usepackage{xspace}
\usepackage{eqparbox}

\usepackage{microtype}

\newcommand\todo[1]{\textcolor{red}{#1}}
\newcommand{\Gpd}{{\bf Gpd}}
\newcommand{\Set}{{\bf Set}}
\newcommand{\Cat}{{\bf Cat}}
\newcommand{\Dbl}{{\bf Dbl}}

\newcommand{\CLift}{\cat{Lift}}

\newcommand{\A}{\ct{A}}
\newcommand{\B}{\ct{B}}
\newcommand{\C}{\ct{C}}
\newcommand{\D}{\ct{D}}
\newcommand{\E}{\ct{E}}

\newcommand{\J}{\mathcal{J}}
\newcommand{\K}{\mathcal{K}}
\renewcommand{\DJ}{\mathbb{J}}
\newcommand{\DK}{\mathbb{K}}
\DeclareMathOperator{\dom}{dom}
\DeclareMathOperator{\cod}{cod}

\newcommand{\DC}{\mathbb{C}}
\newcommand{\DD}{\mathbb{D}}
\newcommand{\Ar}{\cat{Ar}}
\newcommand{\Sq}{\mathbb{S}\cat{q}}

\newcommand{\pf}{\pitchfork}
\newcommand{\dpf}{{\pitchfork\mskip-9.7mu\pitchfork}}
\newcommand{\ldpf}{{^{\pitchfork\mskip-9.7mu\pitchfork}}}
\newcommand{\lpf}{{^\pf}}

\newcommand{\f}{\mathbf{f}}
\newcommand{\g}{\mathbf{g}}
\newcommand{\h}{\mathbf{h}}

\newcommand{\uv}{\mathbf{uv}}
\renewcommand{\i}{\mathbf{i}}
\renewcommand{\j}{\mathbf{j}}

\renewcommand{\u}{\mathbf{u}}
\renewcommand{\v}{\mathbf{v}}
\newcommand{\X}{\mathbf{X}}
\newcommand{\Y}{\mathbf{Y}}
\newcommand{\Q}{\mathbf{Q}}
\renewcommand{\P}{\mathbf{P}}
\newcommand{\w}{\mathbf{w}}

\newcommand{\To}{\Rightarrow}

\newcommand{\F}{\mathbf{F}}
\newcommand{\G}{\mathbf{G}}
\newcommand{\bu}{\mathbf{u}}

\newcommand{\ba}{\boldsymbol{\alpha}}
\newcommand{\be}{\boldsymbol{\varepsilon}}
\newcommand{\bb}{\boldsymbol{\beta}}
\newcommand{\bn}{\boldsymbol{\eta}}
\newcommand{\bmu}{\boldsymbol{\mu}}

\newcommand{\tladj}{\mathrm{ladj}}
\newcommand{\tradj}{\mathrm{radj}}
\newcommand{\Catladj}{\Cat_\tladj}
\newcommand{\Catradj}{\Cat_\tradj}

\newcommand{\cladj}{\Cat/\Ar(-_\tladj)}
\newcommand{\cradj}{\Cat/\Ar(-_\tradj)}
\newcommand{\dladj}{\Dbl/\Sq(-_\tladj)}
\newcommand{\dradj}{\Dbl/\Sq(-_\tradj)}

\newcommand{\lcan}{\lambda}
\newcommand{\rcan}{\rho}

\newcommand{\SplMono}{\cat{SplMono}}
\newcommand{\SplRef}{\cat{SplRef}}
\newcommand{\DSplRef}{\mathbb S\cat{plRef}}

\newcommand{\SplFib}{\cat{SplFib}}
\newcommand{\Fib}{\cat{Fib}}
\newcommand{\DSplFib}{\mathbb S\cat{plFib}}
\newcommand{\SplOpFib}{\cat{SplOpFib}}
\newcommand{\DSplOpFib}{\mathbb S\cat{plOpFib}}

\newcommand{\awfs}{\textsc{awfs}\xspace}
\newcommand{\sfpo}{\textsc{sfpo}}
\newcommand{\wfs}{\textsc{wfs}}

\newcommand{\R}{\mathsf{R}}
\renewcommand{\L}{\mathsf{L}}

\newcommand{\bR}{\mathbb{R}}
\newcommand{\bL}{\mathbb{L}}
\newcommand{\DR}{\mathbb{R}}
\newcommand{\DL}{\mathbb{L}}
\newcommand{\bC}{\mathbb{C}}

\newcommand{\DCoalg}[1]{\mathsf{#1}\text-\mathbb C\cat{oalg}}
\newcommand{\DAlg}[1]{\mathsf{#1}\text-\mathbb A\cat{lg}}

\newcommand{\DLcoalg}{\DCoalg{L}}

\newcommand{\DRalg}{\DAlg{R}}

\usepackage{xcolor, soul}
\sethlcolor{cyan!20!}

\theoremstyle{plain}
\newtheorem{Thm}{Theorem}
\newtheorem{Prop}[Thm]{Proposition}

\newtheorem{Cor}[Thm]{Corollary}
\newtheorem{Lemma}[Thm]{Lemma}
\numberwithin{equation}{section}

\theoremstyle{definition}
\newtheorem{Defn}[Thm]{Definition}

\newtheorem{Ex}[Thm]{Example}

\newtheorem{Rk}[Thm]{Remark}

\date{\today}
\title[The Frobenius Equivalence and Beck--Chevalley Condition for AWFS]{The Frobenius equivalence and Beck--Chevalley condition for Algebraic Weak Factorisation Systems}
\author{Wijnand van Woerkom and Benno van den Berg}

\begin{document}

\begin{abstract}
    If a locally cartesian closed category carries a weak factorisation system, then the left maps are stable under pullback along right maps if and only if the right maps are closed under pushforward along right maps. We refer to this statement as the Frobenius equivalence and in this paper we state and prove an analogical statement for algebraic weak factorisation systems. These algebraic weak factorisation systems are an explicit variant of the more traditional weak factorisation systems in that the factorisation and the lifts are part of the structure of an algebraic weak factorisation system and are not merely required to exist. Our work has been motivated by the categorical semantics of type theory, where the Frobenius equivalence provides a useful tool for constructing dependent function types. We illustrate our ideas using split fibrations of groupoids, which are the backbone of the groupoid model of Hofmann and Streicher.
\end{abstract}

\maketitle

\section{Introduction}

In this paper we contribute to theory of \emph{algebraic weak factorisation systems} (or $\awfs$s, for short). The notion of an $\awfs$ is due to Garner \cite{garner2009UnderstandingSmall}, building on earlier ideas by Grandis and Tholen \cite{grandis2006NaturalWeak}. The notion of an $\awfs$ is a refinement of the well-known notion of a weak factorisation system ($\wfs$), which has become quite important in homotopy theory and category theory. In particular, Quillen's influential notion of a model category involves two weak factorisation systems which interact in a suitable manner \cite{hirschhorn2003ModelCategories}.

The notion of an algebraic weak factorisation system can be seen as a more explicit version of the more familiar notion of a $\wfs$. To explain this, let us recall that a $\wfs$ on a category \ct{C} consists of two classes of maps in \ct{C}, let us denote them by $\mathcal{L}$ and $\mathcal{R}$, which have the following properties:
\begin{enumerate}
    \item Any map in \ct{C} factors as a map in $\mathcal{L}$ followed by a map in $\mathcal{R}$.
    \item Any map in $\mathcal{L}$ lifts again any map in $\mathcal{R}$, in the sense that in any solid commutative square 
    \begin{displaymath}
        \begin{tikzcd}
            \bullet \ar[r] \ar[d, "\in {\mathcal{L}}"'] & \bullet \ar[d, "{\in \mathcal{R}}"] \\
            \bullet \ar[ur, dotted] \ar[r] & \bullet
        \end{tikzcd}
    \end{displaymath}
   in which the map on the left belongs to $\mathcal{L}$ and the map on the right to $\mathcal{R}$, there exists a dotted arrow (the \emph{lift}) as shown, making both triangles commute.
   \item Maps in $\mathcal{L}$ and $\mathcal{R}$ are closed under retracts.
\end{enumerate} 
The theory of $\awfs$s is more ``explicit'' in that when we are given an $\awfs$, we will have explicit witnesses for the truth of axioms (1) and (2): that is, for an $\awfs$ both the factorisation in (1) as well as the lift in (2) can be obtained in an explicit manner and are not just assumed to exist. However, in (2) the lift is given as a function which includes explicit witnessing information for the fact that the map on the left belongs to $\mathcal{L}$ and the map on the right to $\mathcal{R}$. Indeed, when one is given an $\awfs$, this involves both a monad $\R = (R, \eta, \mu)$ and a comonad $\L = (L, \varepsilon, \delta)$ on the category of arrows of \ct{C}. The \emph{property} of being an $\mathcal{L}$-map is then replaced by the \emph{structure} of being a coalgebra for the comonad $\L$ and the \emph{property} of being an $\mathcal{R}$-map is then replaced by the \emph{structure} of being an algebra for the monad $\R$; and it is terms of this structure that the lift in (2) can be constructed. 

The theory of algebraic weak factorisation systems has a number of advantages over the usual theory:
\begin{enumerate}
    \item Due to its explicit nature, it is more constructive than the usual theory.
    \item In an $\awfs$, the coalgebras are closed under colimits and the algebras under limits, in an appropriate sense.
    \item They can be used to describe situations that are not $\wfs$s, because algebras and coalgebras need not be closed under retracts. For example, split Grothendieck fibrations are the algebras for a monad in an $\awfs$, while not being closed under retracts.
\end{enumerate}
For the purposes of this paper, points (1) and (3) are the important ones. Indeed, we feel that point (1) has recently become especially important due to the influence of \emph{homotopy type theory} (HoTT) \cite{theunivalentfoundationsprogram2023HomotopyType}. Here, type theory refers to formal systems, like the Calculus of Constructions or Martin-L\"of's intuitionistic type theory, which are both constructive foundations for mathematics and functional programming languages. Recent developments have amply demonstrated how ideas from homotopy theory and higher category theory are highly relevant to type theory. In particular, Voevodsky has shown how one can exploit the Kan--Quillen model structure on simplicial sets to construct a model of type theory there which validates principles like his Univalence Axiom or Higher Inductive Types \cite{kapulkin2021SimplicialModel}. 

This paper focuses on an important aspect of model constructions like those of Voevodsky which make use of Quillen's model categories to obtain models of type theory: how they model $\Pi$-types. $\Pi$-types, or dependent function types, are arguably the most important type former in type theory and the argument for why these model categories model them is basically always as follows. The underlying category is a presheaf category and hence locally cartesian closed: this means that pullback functors have right adjoints, which we will call \emph{pushforward functors}. To show that we have a model of $\Pi$-types, we need that pushforward functors along $\mathcal{R}$-maps (fibrations) preserve $\mathcal{R}$-maps (fibrations). Using general facts about adjoints (as well as axiom (3) for a \wfs), this is equivalent to the pullback functors along $\mathcal{R}$-maps preserving $\mathcal{L}$-maps. We will refer to this equivalence as the \emph{Frobenius equivalence} and the latter property as the \emph{Frobenius property}. The Frobenius property is often easier to check than the original statement; for instance, the Frobenius property will hold in a model category if it is right proper and the cofibrations are stable under pullback.

The main contribution of this paper is to find suitable analogues of these statement for algebraic weak factorisation systems. That is, we will formulate appropriate forms of the statement that algebras are closed under pushforward along algebras and the Frobenius property for $\awfs$s, and we will prove that these are equivalent. In order to model type theory, one needs suitably ``stable'' versions of both the pushforward and the Frobenius property. This amounts to requiring an additional property, a version of the well-known Beck--Chevalley condition, and we will also formulate an appropriate version of the Beck--Chevalley conditions for $\awfs$s. 

To prove the various equivalences, the connection between $\awfs$s and double categories will be crucial. The way double categories come in is that as double categories the algebras and coalgebras in an $\awfs$ still determine each other. To explain this a bit better, recall that the algebras and coalgebras are not closed under retracts in an $\awfs$ as axiom (3) for a $\wfs$ would suggest. However, axiom (3) is equivalent to another axiom. To formulate it, let us say that $f$ has the right lifting property (RLP) against $g$ and $g$ the left lifting property (LLP) against $f$, denoted $f \pitchfork g$, if for any solid commutative square
\begin{displaymath}
    \begin{tikzcd}
        \bullet \ar[r] \ar[d, "g"'] & \bullet \ar[d, "f"] \\
        \bullet \ar[ur, dotted] \ar[r] & \bullet
    \end{tikzcd}
\end{displaymath}
there exists a dotted lift as shown. If $\mathcal{A}$ is some class of maps in a category \ct{C}, we will write:
\begin{eqnarray*}
    \mathcal{A}^\pitchfork & = & \{ f \, : \, (\forall g \in \mathcal{A}) \, g \pitchfork f \}, \\
    {}^\pitchfork\mathcal{A} & = & \{ g \, : \, (\forall f \in \mathcal{A}) \, g \pitchfork f \}.
\end{eqnarray*}
In those terms, axiom (3) is equivalent (over axioms (1) and (2)) to:
\begin{enumerate}
    \item[(3)$'$] ${}^\pitchfork \mathcal{R} \subseteq \mathcal{L}$ and $\mathcal{L}^\pitchfork \subseteq \mathcal{R}$.
\end{enumerate} 
(So axioms (2) and (3) together are equivalent to ${}^\pitchfork \mathcal{R} = \mathcal{L}$ and $\mathcal{L}^\pitchfork = \mathcal{R}$.) The fact that $\mathcal{L}$ and $\mathcal{R}$ determine each other through lifting properties is very useful and, indeed, this is what is used in the proof of the Frobenius equivalence. While we no longer have that algebras and coalgebras are closed under retracts, we do still have that they determine each other through lifting properties, but to state these we need the language of double categories. 

As it happens, in recent work \cite{bourke2023OrthogonalApproach}, John Bourke has shown that the notion of an \awfs can be formulated purely in terms of double categories and a double-categorical lifting property. It is this definition (called a lifting \awfs there) that is especially useful for this paper and we will use that definition, and not the equivalent one using monads and comonads, to prove the Frobenius equivalence for $\awfs$s.

In this paper we develop this Frobenius equivalence for {\awfs}s in a step-by-step manner as it allows us to gradually introduce the main technical and conceptual points of this paper. Indeed, it has been our experience that by conceptualising matters in the right way we can tell a clear story while also avoiding a lot of messy calculations. For that reason, we start this paper by going over the Frobenius equivalence for weak factorisation systems in some detail in Section 2. This material is well-known, but we will develop it in a particular way, as this allows us to smoothly transition to the more technical setting of an algebraic weak factorisation system. In particular, we emphasise that it is more convenient to think of pullback and pushforward as operations acting on the maps of a slice category (the ``arrow view'') rather than the objects of a slice category (the ``object view'). Before developing a Frobenius equivalence for double categories of maps, we first develop a Frobenius equivalence for categories of maps in Section 3. After that, we define double categories and (lifting) {\awfs}s in Section 4 and state and prove the Frobenius equivalence for double categories of maps and {\awfs}s in Section 5, building on the work in Section 3. For obtaining models of type theories with dependent function types, we need another stability condition for these $\Pi$-types: the Beck-Chevalley condition. We take a closer look at this condition in Sections 6 and formulate an equivalent condition that should make it easier to verify in concrete cases. After a digression in Section 7, where we discuss a strengthening of the Frobenius condition one also encounters in the literature, we show how our framework can be used to obtain a model of type theory with dependent function types, as well as other type formers, such as identity and dependent sum types. We illustrate these ideas using split fibrations of groupoids, the backbone of the groupoid model of Hofmann and Streicher~\cite{hofmann1998GroupoidInterpretation}.

This paper is based on the MSc thesis of the first author, which was supervised by the second author \cite{vanwoerkom2021AlgebraicModels}.
\section{The Frobenius equivalence for classes of maps}\label{sec:subclasses}

In this section we will go over the Frobenius equivalence for weak factorisation systems in some detail. We do this, not because this material is novel (it is not), but because we found that by thinking about this material in a certain way, the generalisation to settings like that of algebraic weak factorisation systems becomes much more natural. Therefore we start this section by recalling the definition of a weak factorisation system and outlining a perspective on the Frobenius equivalence. We will then go on to give a detailed proof of the Frobenius equivalence for weak factorisation systems following this outline.

\subsection{Weak factorisation systems}\label{sec:wfs}
In what follows we frequently work in the context of some (small) ambient category $\C$ which is (sufficiently) locally cartesian closed, and write $\C_1$ for the set (or class) of morphisms in $\C$. Of central importance will be various notions of \emph{lifting}. Given two morphisms $f, g \in \C_1$ we say $f$ has the left lifting property with respect to $g$, and $g$ the right lifting property with respect to $f$, denoted $f \pitchfork g$, if every commutative square $(u, v) : f \to g$ has a diagonal filler $\phi : \cod f \to \dom g$ making both of the induced triangles commute: 
\begin{equation}\label{eq:phi}
    \begin{tikzcd}
        \bullet & \bullet \\
        \bullet & {\bullet \rlap{ .}}
        \arrow["f"', from=1-1, to=2-1]
        \arrow["g", from=1-2, to=2-2]
        \arrow["u", from=1-1, to=1-2]
        \arrow["v"', from=2-1, to=2-2]
        \arrow["\phi", dotted, from=2-1, to=1-2]
    \end{tikzcd}
\end{equation}
Given a subclass $J \subseteq \C_1$ of morphisms in $\C$ we define \begin{align*}
    {^\pf}J &= \{f \in \C_1 \mid f \pitchfork g \text{ for all $g \in J$}\}, \text{ and} \\
    J\lpf &= \{f \in \C_1 \mid g \pitchfork f \text{ for all $g \in J$}\}.
\end{align*}
\begin{Defn}\label{def:wfs}
    A \emph{weak factorisation system} (\wfs) on $\C$ is a pair $(L, R)$ of subclasses $L, R \subseteq \C_1$ satisfying two axioms:
    \begin{enumerate}
        \item \emph{Axiom of lifting}: $L = \lpf R$ and $R = L \lpf$.
        \item \emph{Axiom of factorisation}: every morphism $f \in \C_1$ admits a factorisation $f = r.l$ with $l \in L$ and $r \in R$. 
    \end{enumerate}
    The set $L$ is referred to as the \emph{left} class of the \wfs, and $R$ as the \emph{right} class. 
\end{Defn}
\begin{Rk}
    Throughout this work we follow the notational conventions of~\cite{bourke2016AlgebraicWeak,bourke2023OrthogonalApproach}, in using the letters $J$ and $K$ to denote arbitrary classes of morphisms in a category, and the letters $L$ and $R$ to denote the classes making up a weak factorisation system. However, we will often also consider general pairs $(J, K)$ of subclasses of $\C_1$, and also then we will speak of left and right maps to refer to morphisms from $J$ and $K$, respectively.
\end{Rk}

\subsection{The Frobenius equivalence in outline}\label{sec:outlineFrobenius}

When $\C$ is locally cartesian closed every morphism $f \in \C_1$ induces an adjunction $f^* \dashv f_*$ of pullback and pushforward along $f$, respectively. The two properties of a \wfs\ on such a category that we will be concerned with in this paper are the following. 
\begin{Defn}\label{def:wfsfrob}
    A \wfs\ $(L, R)$ on a locally cartesian closed category $\C$ satisfies the \emph{Frobenius property} if $L$ is closed under pullback along $R$, and the \emph{pushforward property} if $R$ is closed under pushforward along $R$. 
\end{Defn}

\begin{Prop}[Frobenius equivalence for a \wfs]\label{prop:wfsfrobequiv}
    A \wfs\ satisfies the Frobenius property if and only if it satisfies the pushforward property. 
\end{Prop}

The statement of the Frobenius equivalence of Proposition~\ref{prop:wfsfrobequiv} can be understood as the result of combining the following two well-known results on \wfs. 
\begin{Prop}\label{prop:wfsbasechange}
    Let $F \dashv G : \D \to \C$ be an adjunction, $(L, R)$ a \wfs\ on $\C$, and $(L', R')$ one on $\D$; then $F(L) \subseteq L'$ if and only if $G(R') \subseteq R$. 
\end{Prop}
This statement roughly says that a left adjoint preserves left maps if and only if its right adjoint preserves right maps. Following the terminology of~\cite{bourke2016AlgebraicWeak} we will refer to this and similar results as \emph{change of base}.

\begin{Prop}\label{prop:slicewfs}
    Given a \wfs\ $(L, R)$ on a category $\C$ and an object $A \in \C$ there is a \wfs\ $(L/A, R/A)$ on the slice category $\C/A$. 
\end{Prop}
The \wfs\ produced by Proposition~\ref{prop:slicewfs} are often called \emph{slice} \wfs. The left and right maps of the slice \wfs\ on $\C/A$ are those maps in $\C/A$ whose underlying morphisms are left or right maps, respectively. 

Proposition~\ref{prop:wfsbasechange} now combines with Proposition~\ref{prop:slicewfs} to yield a slightly more general version of the Frobenius equivalence. 
\begin{Cor}\label{cor:wfsFE}
    Let $f : A \to B$ be a morphism in $\C$ and $(L, R)$ a \wfs\ on $\C$, then $f^*(L/B) \subseteq L/A$ if and only if $f_*(R/A) \subseteq R/B$. 
\end{Cor}
\begin{proof}
    By Proposition~\ref{prop:slicewfs} we obtain two slice \wfs\ $(L/A, R/A)$ and $(L/B, R/B)$ on $\C/A$ and $\C/B$ respectively. We have an adjunction $f^* \dashv f_*: \C/A \to \C/B$ and so Proposition~\ref{prop:wfsbasechange} tells us $f^*(L/B) \subseteq L/A$ if and only if $f_*(R/A) \subseteq R/B$. 
\end{proof}
\begin{Rk}
    While Proposition~\ref{prop:wfsfrobequiv} and Corollary~\ref{cor:wfsFE} are clearly very similar, there are still some notable differences between them; for example, Proposition~\ref{prop:wfsfrobequiv} is about base change along right maps, whereas in Corollary~\ref{cor:wfsFE} the base change is performed along an arbitrary morphism. We will get back to these differences in Section~\ref{sec:sc:frobequiv}.
\end{Rk}

As a matter of fact, the proofs of Propositions~\ref{prop:wfsbasechange}~and~\ref{prop:slicewfs} do not rely on the axiom of factorisation of \wfs, but only on the axiom of lifting. They can therefore be stated for what are called \emph{pre}-\wfs\ in~\cite{bourke2023OrthogonalApproach, freyd1972CategoriesContinuous}, but which we will call \emph{closed lifting pairs}, i.e.\ pairs $(J, K)$ satisfying $J = \lpf K$ and $K = J\lpf$. We thus see that the Frobenius equivalence of Corollary~\ref{cor:wfsFE} consists of the following key components. 
\begin{enumerate}
    \item\label{item:lifting} \emph{Lifting}: A definition of lifting with respect to a subclass $J \subseteq \C_1$, giving rise to subclasses of left lifting $\lpf J \subseteq \C_1$ and right lifting $J \lpf \subseteq \C_1$. In turn, this yields the notion of a closed lifting pair $(J, K)$ satisfying $J = \lpf K$ and $K = J\lpf$. 
    \item \label{item:basechange} \emph{Base change}: A result relating closed lifting pairs along adjunctions. 
    \item\label{item:slicing} \emph{Slicing}: An extension of the slice category construction $\C/A$ for some object $A \in \C$ to a subclass $J \subseteq \C_1$, yielding a subclass $J/A \subseteq (\C/A)_1$. This construction should preserve closed lifting pairs, meaning that when $(J, K)$ is a closed lifting pair, then so is $(J/A, K/A)$. 
\end{enumerate}
Once these components have been established, definitions of the Frobenius and pushforward properties, as well as the equivalence between them, follow readily.

We will now consider components (\ref{item:lifting})--(\ref{item:slicing}) of lifting, base change, and slicing in some detail in Sections~\ref{sec:sslifting}, \ref{sec:ssbasechange}, and \ref{sec:ssslicing}, respectively. We conclude with the Frobenius equivalence for weak factorisation systems in Section~\ref{sec:sc:frobequiv}. 

\subsection{Lifting}\label{sec:sslifting}

As mentioned in Section~\ref{sec:wfs}, we write $f \pitchfork g$ for $f, g \in \C_1$ when every lifting problem $(u, v) : f \to g$ has a diagonal filler $\phi$ as in~(\ref{eq:phi}). We extend this notation to subclasses $J, K \subseteq \C_1$, by writing $J \pitchfork K$ when $f \pitchfork g$ for every $f \in J$ and $g \in K$; for singletons we omit brackets. We define, as before, $\lpf J = \{f \in \C_1 \mid f \pitchfork J\}$ and $J \lpf = \{f \in \C_1 \mid J \pitchfork f\}$. Let $\P(\C_1)$ denote the poset category of subclasses of $\C_1$ ordered by the inclusion relation, then $\lpf(-)$ and $(-)\lpf$ give a \emph{Galois connection}: 
\begin{equation}\label{eq:sc:galois}
\begin{tikzcd}[ampersand replacement=\&]
	{(\P(\C_1))^\text{op}} \& {\P(\C_1) \rlap{ .}}
	\arrow[""{name=0, anchor=center, inner sep=0}, "{\lpf(-)}"', shift right=2, from=1-2, to=1-1]
	\arrow[""{name=1, anchor=center, inner sep=0}, "{(-)\lpf}"', shift right=2, from=1-1, to=1-2]
	\arrow["\dashv"{anchor=center, rotate=-90}, draw=none, from=0, to=1]
\end{tikzcd}
\end{equation}
In this notation we have a string of equivalences:
\begin{alignat}{10}\label{eq:ssJpfK}
    J \subseteq \lpf K &\quad\text{ iff }\quad& J \pitchfork K &\quad\text{ iff }\quad& K \subseteq J \lpf.
\end{alignat}
We are particularly interested in cases where the inclusions of~(\ref{eq:ssJpfK}) are equalities. 
\begin{Defn}\label{def:closedliftingpair}
    A pair $(J, K)$ of subclasses $J, K \subseteq \C_1$ is called a \emph{lifting pair} if $J \pitchfork K$. A lifting pair $(J, K)$ is called \emph{closed} if the induced inclusions of (\ref{eq:ssJpfK}) hold both ways, so that $J = \lpf K$ and $K = J \lpf$. 
\end{Defn}
\begin{Rk}
    In~\cite[Definition~1.1.3]{north2017TypeTheoretic} a lifting pair is defined as what we call a closed lifting pair in Definition~\ref{def:closedliftingpair}. We make this further distinction to align terminology with~\cite{bourke2023OrthogonalApproach}, cf.\ Definition~\ref{def:funcliftingstructure}. 
\end{Rk}

\begin{Ex}\label{ex:sscanliftingpairs}
    We remark, as in~\cite[Examples 5 \& 6]{bourke2023OrthogonalApproach}, that any subclass $J \subseteq \C_1$ induces two canonical lifting pairs $(\lpf J, J)$ and $(J, J \lpf)$, because from~(\ref{eq:ssJpfK}) we get:
    \begin{alignat*}{7}
        \lpf J &\subseteq \lpf J &\quad\text{ iff }\quad&& \lpf J &\pitchfork J &&\quad\text{ iff }\quad& J &\subseteq (\lpf J)\lpf, \text{ and}\\
        J &\subseteq \lpf(J \lpf) &\quad\text{ iff }\quad&& J &\pitchfork J \lpf &&\quad\text{ iff }\quad& J \lpf &\subseteq J \lpf.
    \end{alignat*}
    The inclusions $J \subseteq (\lpf J)\lpf$ and $J \subseteq \lpf(J \lpf)$ correspond to the unit and counit of~(\ref{eq:sc:galois}). The lifting pairs $(\lpf J, (\lpf J)\lpf)$ and $(\lpf(J\lpf), J\lpf)$ so obtained are called \emph{fibrantly generated} and \emph{cofibrantly generated} by $J$, respectively. In fact, since~(\mbox{\ref{eq:sc:galois}}) is a Galois connection, these fibrantly and cofibrantly generated pairs are always closed.
\end{Ex}

\subsection{Base change}\label{sec:ssbasechange}

The notion of lifting pair can be generalized relative to adjunctions of categories. This is because for $J \subseteq \C_1$ and $K \subseteq \D_1$, with an adjunction $F \vdash G$ between two categories $\C$ and $\D$, we have the following string of equivalences:
\begin{alignat}{100}\label{eq:ssJpfKadj}
    F(J) \subseteq \lpf K &\quad\text{ iff }\quad& F(J) \pitchfork K &\quad\text{ iff }\quad& J \pitchfork G(K) &\quad\text{ iff }\quad& G(K) \subseteq J \lpf.
\end{alignat}

\begin{Prop}\label{prop:sc:basechange}
    Let $F \dashv G : \D \to \C$ be an adjunction, and $J \subseteq \C_1$, $K \subseteq \D_1$ classes of maps in $\C$ and $\D$ respectively; then $F(J) \subseteq \lpf K$ if and only if $G(K) \subseteq J \lpf$.  
\end{Prop}
\begin{proof}
    We only prove left to right; the other direction is dual. To show $Gk \in J\lpf$ for $k : A \to B \in K$ we consider $j : C \to D \in J$ and $(u, v) : j \to Gk$ as on the left of: 
\[\begin{tikzcd}
	C & GA && FC & A \\
	D & {GB \rlap{ ,}} && FD & {B \rlap{ .}}
	\arrow["Gk", from=1-2, to=2-2]
	\arrow["j"', from=1-1, to=2-1]
	\arrow["u", from=1-1, to=1-2]
	\arrow["Fj"', from=1-4, to=2-4]
	\arrow["k", from=1-5, to=2-5]
	\arrow["{\bar{u}}", from=1-4, to=1-5]
	\arrow["{\bar{v}}"', from=2-4, to=2-5]
	\arrow["\phi", dotted, from=2-4, to=1-5]
	\arrow["v"', from=2-1, to=2-2]
	\arrow["{\bar{\phi}}", dotted, from=2-1, to=1-2]
\end{tikzcd}\]
Transposing gives a lifting problem $(\bar{u}, \bar{v})$ for $Fj$ and $k$ which has a solution $\phi$ by the assumption $F(J) \subseteq \lpf K$, and its transpose $\smash{\bar{\phi}}$ is a solution to $(u, v)$.
\end{proof}
In the context of closed lifting pairs Proposition~\ref{prop:sc:basechange} tells us that a left adjoint preserves left maps if and only if its right adjoint preserves right maps. 
\begin{Cor}\label{cor:ssbasechange}
    Let $F \dashv G : \D \to \C$, and $(J, K)$ and $(L, M)$ be closed lifting pairs on $\C$ and $\D$ respectively, then $F(J) \subseteq L$ if and only if $G(M) \subseteq K$.  
\end{Cor}

\subsection{Slicing}\label{sec:ssslicing}

A large part of our work relates to slicing. Recall that if $\C$ is a category and $A$ an object in $\C$, then the \emph{slice category} $\C/A$ is the category whose objects are maps $a: X \to A$ with codomain $A$ and whose morphisms $(b: Y \to A) \to (a: X \to A)$ are maps $f: Y \to X$ in $\C$ making the triangle below commute:
\[\begin{tikzcd}
	Y && {X } \\
	& {A \rlap{ .}}
	\arrow["f", from=1-1, to=1-3]
	\arrow["b"', from=1-1, to=2-2]
	\arrow["a", from=1-3, to=2-2]
\end{tikzcd}\]
Note that $f$ (considered as a morphism in $\C$) does not uniquely determine a morphism in $\C/A$, as there could be other arrows besides $a$ in $\C(X, A)$. On the other hand, the domain $b$ of $f$ (considered as a morphism in $\C/A$) can be inferred from the data $f$ and $a$. This means an arrow in $\C/A$ is equivalently a pair $(f, a)$ of composable morphisms such that $\cod a = A$, and we will henceforth refer to them using this notation. In this setting we say $a$ is an \emph{extension} of $f$ to $A$. Two morphisms $(f, a)$ and $(g, b)$ in $\C/A$ are composable when $a = b.g$, and their composition is given by $(g.f, b)$. 

Next, we extend the slicing operation on categories to one on classes of maps.
\begin{Defn}\label{def:inclslicing}
    Given a subclass $J\subseteq \C_1$ and an object $A \in \C$ we define the \emph{slice} subclass $J/A \subseteq (\C/A)_1$ by $J/A = \{(f, a) \in (\C/A)_1 \mid f \in J\}$. 
\end{Defn} 

The goal will now be to prove the following proposition, which uses the above notion of slicing to create new closed lifting pairs from old. 
\begin{Prop}\label{prop:sc:sliceclp}
    If $(J, K)$ is a closed lifting pair on $\C$, and $A \in \C$ is some object, then $(J/A, K/A)$ is a closed lifting pair on $\C/A$. 
\end{Prop}

We start by considering how lifting works in a slice category. A lifting problem for morphisms $(f, a)$ and $(g, b)$ in $\C/A$ is a square $((u, c), (v, d)) : (f, a) \to (g, b)$, as in:
\begin{equation}\label{eq:slicelift}
    \begin{tikzcd}[sep=scriptsize]
        {a.f} && {b.g} &&& \bullet && \bullet \\
        \\
        a && {b \rlap{ ,}} &&& \bullet && \bullet \\
        &&&&&& {A \rlap{ .}}
        \arrow["f"', from=1-6, to=3-6]
        \arrow["g", from=1-8, to=3-8]
        \arrow["u", from=1-6, to=1-8]
        \arrow["v"', from=3-6, to=3-8]
        \arrow["\phi", dotted, from=3-6, to=1-8]
        \arrow["a"', from=3-6, to=4-7]
        \arrow["b", from=3-8, to=4-7]
        \arrow["{(f, a)}"', from=1-1, to=3-1]
        \arrow["{(g, b)}", from=1-3, to=3-3]
        \arrow["{(u, c)}", from=1-1, to=1-3]
        \arrow["{(v, d)}"', from=3-1, to=3-3]
        \arrow["{(\phi, e)}", dotted, from=3-1, to=1-3]
    \end{tikzcd}
\end{equation}
This boils down to a lifting problem $(u, v) : f \to g$ for $f$ and $g$ such that the induced triangle over the extensions to $A$ commutes, as illustrated on the right above. Moreover, a lifting solution $(\phi, e)$ to the problem $((u, c), (v, d))$ is just a solution $\phi$ to the problem $(u, v)$, as $e$ necessarily equals $b.g$. In other words, lifting in $\C/A$ is very similar to lifting in $\C$, with the difference being that in $\C/A$ a restricted set of lifting problems is considered---namely those commuting with the extensions to $A$.

\begin{Lemma}\label{lem:slicelift}
    Let $(f, a), (g, b) \in (\C/A)_1$ for $A \in \C$, then $f \pitchfork g$ implies $(f, a) \pitchfork (g, b)$. 
\end{Lemma}
\begin{proof}
    A lifting problem $((u, c), (v, d))$ for $(f, a)$ and $(g, b)$ as on the left of~(\ref{eq:slicelift}) yields a problem $(u, v)$ for $f$ and $g$ as on the right of~(\ref{eq:slicelift}), which has a solution $\phi$ by assumption; now $(\phi, b.g)$ is a solution to $((u, c), (v, d))$. 
\end{proof}

\begin{Lemma}\label{lem:pf/A}
    Let $J, K \subseteq \C_1$ and $A \in \C$, then $J \pitchfork K$ implies $J/A \pitchfork K/A$. 
\end{Lemma}
\begin{proof}
    Let $(f, a) \in J/A, (g, b) \in K/A$, then $f \in J$ and $g \in K$ so $f \pitchfork g$ by assumption, and so $(f, a) \pitchfork (g, b)$ by Lemma~\ref{lem:slicelift}.
\end{proof}

Lemma~\ref{lem:pf/A} tells us the slicing operation on subclasses of maps preserves lifting pairs. To see that it also preserves closed lifting pairs we appeal to a more general statement about commutativity between slicing and lifting. The first step towards proving commutativity of lifting is taken by the following consequence of Lemma~\ref{lem:pf/A}.
\begin{Lemma}\label{lem:subclass:sc}
    Let $J \subseteq \C_1$ and $A \in \C$, then $(J\lpf)/A \subseteq (J/A)\lpf$ and $(\lpf J)/A \subseteq \lpf(J/A)$.
\end{Lemma}
\begin{proof}
    By applying Lemma~\ref{lem:pf/A} to the canonical lifting pairs $(\lpf J, J)$ and $(J, J\lpf)$ we get $(\lpf J) / A \pitchfork J / A$ and $J / A \pitchfork (J\lpf) / A$, which is equivalent to the desired inclusions. 
\end{proof}
Only one of the converse inclusions of Lemma~\ref{lem:subclass:sc} holds in general.
\begin{Lemma}\label{lem:inclslicingcomrl}
    Let $J \subseteq \C_1$ and $A \in \C$, then $(J/A)\lpf \subseteq (J\lpf)/A$.
\end{Lemma}
\begin{proof}
    Let $(f, a) \in (J/A)\lpf$, to show $f \in J\lpf$ we consider $(u, v) : j \to f$ for $j \in J$, as in:
    
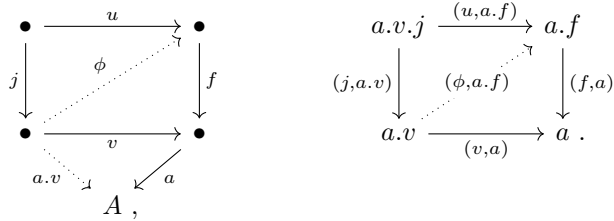
\[\begin{tikzcd}[sep=scriptsize]
        \bullet && \bullet &&& {a.v.j} && {a.f} \\
        \\
        \bullet && {\bullet } &&& {a.v} && {a \rlap{ .}} \\
        & {A \rlap{ ,}}
        \arrow["j"', from=1-1, to=3-1]
        \arrow["f", from=1-3, to=3-3]
        \arrow["u", from=1-1, to=1-3]
        \arrow["v"', from=3-1, to=3-3]
        \arrow["\phi", dotted, from=3-1, to=1-3]
        \arrow["{a.v}"', dotted, from=3-1, to=4-2]
        \arrow["a", from=3-3, to=4-2]
        \arrow["{(j,a.v)}"', from=1-6, to=3-6]
        \arrow["{(f,a)}", from=1-8, to=3-8]
        \arrow["{(u, a.f)}", from=1-6, to=1-8]
        \arrow["{(v, a)}"', from=3-6, to=3-8]
        \arrow["{(\phi, a.f)}"{description}, dotted, from=3-6, to=1-8]
    \end{tikzcd}\]
    Now $(j, a.v) \in J/A$, and we have a commuting square $((u, a.f), (v, a)) : (j, a.v) \to (f, a)$ in $\C/A$ which has a solution $(\phi, a.f)$; and $\phi$ is a solution to the square $(u, v)$.
\end{proof}
\begin{Cor}\label{cor:inclcomll}
    Let $J \subseteq C_1$, then $(J/A) \lpf = (J\lpf)/A$.
\end{Cor}

The trick in the proof of Lemma~\ref{lem:inclslicingcomrl} is to make an extension for $j$ by precomposing the extension of $f$ with the lower half of the lifting problem. This does not work when we try to prove the remaining inclusion $\lpf(J/A) \subseteq (\lpf J)/A$, which involves a square $(u, v) : f \to j$ in which the positions of $f$ and $j$ are swapped. To work around this, we place an additional requirement on $J$. 
\begin{Lemma}\label{lem:inclslicingcomrlagain}
    Let $J \subseteq \C_1$ be closed under pullbacks and $A \in \C$, then $\lpf(J/A) \subseteq (\lpf J)/A$.
\end{Lemma}
\begin{proof}
    Let $(f, a) \in \lpf(J/A)$, to show $f \in J\lpf$ we consider $(u, v) : f \to j$ for $j \in J$, as in:
\[\begin{tikzcd}[sep={.5cm}]
	\bullet && \bullet && \bullet &&& {a.f} && {a.v^*j} \\
	\\
	\bullet && \bullet && {\bullet \rlap{ ,}} &&& a && {a \rlap{ .}} \\
	& A && \phantom{A}
	\arrow["f"', from=1-1, to=3-1]
	\arrow["j", from=1-5, to=3-5]
	\arrow["u", curve={height=-18pt}, from=1-1, to=1-5]
	\arrow["a"', from=3-1, to=4-2]
	\arrow["{(f, a)}"', from=1-8, to=3-8]
	\arrow["{(v^*j, a)}", from=1-10, to=3-10]
	\arrow["{(\alpha, a.v^*j)}", from=1-8, to=1-10]
	\arrow["{(1, a)}"', from=3-8, to=3-10]
	\arrow["{(\phi, a.v^*j)}"{description}, dotted, from=3-8, to=1-10]
	\arrow["v"', from=3-3, to=3-5]
	\arrow["1"', from=3-1, to=3-3]
	\arrow["a", from=3-3, to=4-2]
	\arrow["\phi"{description}, dotted, from=3-1, to=1-3]
	\arrow["\alpha"', dotted, from=1-1, to=1-3]
	\arrow["{\varepsilon_v}"', from=1-3, to=1-5]
	\arrow["{v^*j}", from=1-3, to=3-3]
	\arrow["\lrcorner"{anchor=center, pos=0.125}, draw=none, from=1-3, to=3-5]
\end{tikzcd}\]
By assumption $v^*j \in J$ and so $(v^*j, a) \in J/A$. The morphism $\alpha : \dom f \to \dom v^*j$ induced by $v^*j$ gives rise to a lifting problem $((\alpha, a.v^*j), (1, a)) : (f, a) \to (v^*j, a)$ in $\C/A$, which has a solution $(\phi, a.v^*j)$; and $\varepsilon_v.\phi$ is a solution to $(u, v)$. 
\end{proof}
This additional assumption on $J$ is satisfied in the case we are interested in.
\begin{Lemma}\label{lem:sc:pb}
    For $J \subseteq \C_1$ the class $J\lpf$ is closed under pullbacks.
\end{Lemma}
\begin{proof}
    Let $f \in J\lpf$ and $(u, v) : g \to f$ be a pullback. To show $g \in J\lpf$ we consider a lifting problem $(w,x) : j \to g$ for $j \in J$:
\[\begin{tikzcd}[sep=2.5em]
	\bullet & \bullet & \bullet \\
	\bullet & \bullet & \bullet \rlap{ .}
	\arrow["j"', from=1-1, to=2-1]
	\arrow["g"', from=1-2, to=2-2]
	\arrow["f", from=1-3, to=2-3]
	\arrow["u", from=1-2, to=1-3]
	\arrow["v"', from=2-2, to=2-3]
	\arrow["\lrcorner"{anchor=center, pos=0.125}, draw=none, from=1-2, to=2-3]
	\arrow["w", from=1-1, to=1-2]
	\arrow["x"', from=2-1, to=2-2]
	\arrow["\phi"'{pos=0.7}, curve={height=6pt}, dotted, from=2-1, to=1-3]
	\arrow["\alpha", dotted, from=2-1, to=1-2]
\end{tikzcd}\]
By assumption the outer rectangle $(u.w, v.x)$ has a solution $\phi$, by which the pullback induces a morphism $\alpha$ that solves $(w, x)$. 
\end{proof}
\begin{Cor}\label{cor:sscomslicerl}
    Let $J, K \subseteq \C_1$ with $K = J\lpf$, then $\lpf(K/A) = (\lpf K)/A$. 
\end{Cor}

We are now ready to prove that slicing preserves closed lifting pairs. 
\begin{proof}[Proof of Proposition~\ref{prop:sc:sliceclp}]
    Let $(J, K)$ be a closed lifting pair on $\C$; we want to show that $(J/A, K/A)$ is a closed lifting pair on $\C/A$. By Corollaries~\ref{cor:inclcomll} and~\ref{cor:sscomslicerl} we have:
    \[(J/A)\lpf = (J\lpf)/A = K/A, \quad\text{and}\quad \lpf(K/A) = (\lpf K) / A = J / A. \qedhere\]
\end{proof}

\subsection{Frobenius equivalence}\label{sec:sc:frobequiv}

In a locally cartesian closed category $\C$, any morphism $f : A \to B$ induces an adjunction $f^* \dashv f_* : \C/A \to \C/B$ of pullback and pushforward along $f$. The action of the pullback functor $f^*$ on an object $g$ in $\C/B$ is depicted on the left diagram below:
\begin{equation}\label{eq:pullbackfunctor}
    \begin{tikzcd}
        &&&& \bullet & \bullet \\
        \bullet & \bullet &&& \bullet & \bullet \\
        A & {B \rlap{ ,}} &&& A & {B \rlap{ .}}
        \arrow[from=2-5, to=2-6]
        \arrow["b", from=2-6, to=3-6]
        \arrow["g", from=1-6, to=2-6]
        \arrow["f"', from=3-5, to=3-6]
        \arrow["{f^*b}", from=2-5, to=3-5]
        \arrow["{f^*(b.g)}"', curve={height=18pt}, from=1-5, to=3-5]
        \arrow[from=2-5, to=2-6]
        \arrow[from=1-5, to=1-6]
        \arrow["{f^*g}", dotted, from=1-5, to=2-5]
        \arrow["g", from=2-2, to=3-2]
        \arrow["{f^*g}"', from=2-1, to=3-1]
        \arrow["f"', from=3-1, to=3-2]
        \arrow[from=2-1, to=2-2]
        \arrow["\lrcorner"{anchor=center, pos=0.125}, draw=none, from=2-1, to=3-2]
    \end{tikzcd}
\end{equation}
The action of $f^*$ on an arrow $(g, b)$ in $\C/B$ is depicted on the right above; it is defined as the morphism induced by the pullback $f^*b$ from the pullback $f^*(b.g)$ of the composition $b.g$. In this case we (abusively) write $f^*g$ for the first component of the resulting morphism in $\C/A$, so $f^*(g, b) = (f^*g, f^*b)$.

\begin{Defn}\label{def:sc:frob}
    A morphism $f : A \to B \in \C_1$ has the \emph{Frobenius property} with respect to a pair $(J, K)$ of subclasses $J, K \subseteq \C_1$ if $f^*(J/B) \subseteq J/A$, and the \emph{pushforward property} if $f_*(K/A) \subseteq K/B$. 
\end{Defn}
\begin{Thm}\label{thm:sc:frobequiv}
    A morphism $f : A \to B$ in $\C$ has the Frobenius property with respect to a closed lifting pair $(J, K)$ if and only if it has the pushforward property.
\end{Thm}
\begin{proof}
    By Propositions~\ref{prop:sc:basechange}~and~\ref{prop:sc:sliceclp} we have:
    \begin{align*}
        f^*(J/B) \subseteq J/A &\text{ iff } f^*(J/B) \subseteq \lpf(K/A) \\
        &\text{ iff } f_*(K/A) \subseteq \smash{(J/B)\lpf}\\
        &\text{ iff } f_*(K/A) \subseteq K/B. \qedhere 
    \end{align*}
\end{proof}

The definition of the Frobenius property we give in Definition~\ref{def:sc:frob} differs from how it is usually phrased for \wfs\ (Definition~\ref{def:wfsfrob} above), in two ways. Firstly, Definition~\ref{def:wfsfrob} requires the \emph{object} component of the pullback functor $f^*$ to preserve left maps, while in Definition~\ref{def:sc:frob} it is the \emph{arrow} component of $f^*$ which is required to do so. This is illustrated in~(\ref{eq:pullbackfunctor})---both versions of the Frobenius property require that $g \in L$ implies $f^*g \in L$, but Definition~\ref{def:wfsfrob} is phrased with respect to the left diagram of~(\ref{eq:pullbackfunctor}) and Definition~\ref{def:sc:frob} with respect to the right diagram. The latter is slightly stronger. 
\begin{Prop}\label{prop:sc:arrtoobj}
    Let $(J, K)$ be classes of morphisms in $\C$ which are closed under isomorphisms, and $f$ a map in $\C$. If the arrow component of $f^*$ preserves $J$ maps then so does its object component; the same statement holds for $f_*$ w.r.t.\ $K$ maps. 
\end{Prop}
\begin{proof}
    To avoid confusion we write $f^*$ for the action of the pullback functor along $f$ on objects, and $f^\star$ for its action on arrows. Let $g : C \to B$ be a $J$ map, we want to show that $f^*g \in J$. Note that $g$ can be considered a morphism $(g, 1_A)$ in $\C/B$, and so $f^\star g \in J$. Therefore, we have $f^*g \in J$ from $(1, f^*1) : f^\star g \cong f^*g$:
\[\begin{tikzcd}[ampersand replacement=\&]
	\bullet \& \bullet \\
	\bullet \& \bullet \\
	A \& {B \rlap{ .}}
	\arrow["1", from=2-2, to=3-2]
	\arrow["{f^*1}", from=2-1, to=3-1]
	\arrow["f"', from=3-1, to=3-2]
	\arrow["g", from=1-2, to=2-2]
	\arrow["{f^\star g}", from=1-1, to=2-1]
	\arrow[from=1-1, to=1-2]
	\arrow["{f^*g}"', curve={height=18pt}, from=1-1, to=3-1]
	\arrow[from=2-1, to=2-2]
\end{tikzcd}\]
The argument for the pushforward functor $f_*$ is analogous. 
\end{proof}

The converse implications do not hold in general for an individual morphism $f$. To obtain something of a converse we should make the statement for a class of morphisms $K$ which is closed under pullbacks. 
\begin{Prop}\label{prop:sc:objtoarr}
    Let $(J, K)$ be closed under isomorphisms, $K$ closed under pullbacks, and $f$ a map in $K$. If the object component of $f^*$ preserves $J$ maps then so does its arrow component. Furthermore, if the object components of all pushforward functors of $K$ maps preserve $K$ maps, then the arrow component of $f_*$ preserves $K$ maps. 
\end{Prop}
\begin{proof}
    Let $f : A \to B$ be a $K$ map, and $(g, b) \in \C/B$. Writing $\varepsilon$ for the counit of $f_! \dashv f^*$, we have that $\varepsilon_b \in K$ by assumption, and so $\varepsilon_b^*g \cong f^* g \in K$; as on the left of:
\[\begin{tikzcd}[ampersand replacement=\&]
	\bullet \& \bullet \& \bullet \&\& \bullet \& \bullet \& \bullet \& \bullet \\
	\& \bullet \& \bullet \&\& \bullet \& \bullet \& \bullet \\
	\& A \& {B\rlap{ ,}} \&\& A \&\& {B \rlap{ .}}
	\arrow["f"', from=3-2, to=3-3]
	\arrow["b", from=2-3, to=3-3]
	\arrow["g", from=1-3, to=2-3]
	\arrow["{f^*b}", from=2-2, to=3-2]
	\arrow["{\varepsilon_b}"', from=2-2, to=2-3]
	\arrow[""{name=0, anchor=center, inner sep=0}, "{\varepsilon_b^*g}", from=1-2, to=2-2]
	\arrow[from=1-2, to=1-3]
	\arrow[""{name=1, anchor=center, inner sep=0}, "{f^* g}"', from=1-1, to=2-2]
	\arrow[from=1-1, to=1-2]
	\arrow["a"', from=2-5, to=3-5]
	\arrow["f"', from=3-5, to=3-7]
	\arrow["g"', from=1-5, to=2-5]
	\arrow["{f_*a}", from=2-7, to=3-7]
	\arrow["{f^*f_*a}", from=2-6, to=3-5]
	\arrow["{\varepsilon_{f_*a}}", from=2-6, to=2-7]
	\arrow["{\nu_a}"', from=2-6, to=2-5]
	\arrow["{\nu_a^*g}"', from=1-6, to=2-6]
	\arrow[from=1-6, to=1-5]
	\arrow[""{name=2, anchor=center, inner sep=0}, "{f_*g}", from=1-8, to=2-7]
	\arrow[""{name=3, anchor=center, inner sep=0}, "h"', from=1-7, to=2-7]
	\arrow[from=1-7, to=1-8]
	\arrow["\sim", shift left, shorten <=5pt, shorten >=5pt, Rightarrow, no head, from=1, to=0]
	\arrow["\sim", shift left, shorten <=5pt, shorten >=5pt, Rightarrow, no head, from=3, to=2]
\end{tikzcd}\]
Consider $(g, a) \in \C/A$, and let $\nu : f^*f_* \to 1$ denote the counit of $f^* \ladj f_*$. As $K$ is closed under pullbacks, we have that $\nu_a^*g$ and $\varepsilon_{f_*a}$ are both in $K$, and so $(\varepsilon_{f_*a})_*\nu_a^*g = h \in K$. It can be shown that $h$ has the universal property of $f_*g$, and so $h \cong f_*g \in K$.
\end{proof}

Despite the fact that under the assumptions of Proposition~\ref{prop:sc:objtoarr} the Frobenius and pushforward properties for arrows can be reduced to those for objects, we still regard the ones for arrows as being the most fundamental. Indeed, the versions for arrows are the ones that can most readily be generalised to the case of categories and double categories of maps, as we will see in the subsequent sections. 

Proposition~\ref{prop:sc:objtoarr} brings us to a second difference between Definitions~\ref{def:wfsfrob} and~\ref{def:sc:frob}, which is that in the former both the pullback and pushforward is performed along a $K$ map, whereas the latter is phrased with respect to an arbitrary morhpism in the ambient category. There are two reasons why Definition~\ref{def:wfsfrob} is phrased this way. Firstly, as Proposition~\ref{prop:sc:objtoarr} suggests, it is necessary for obtaining Proposition~\ref{prop:wfsfrobequiv} (as the right class of a \wfs\ is always closed under pullbacks). The second reason relates to type theory; we want to interpret dependent types as right maps of a \wfs, and the dependent product as pushforward, so that $R$ should be closed under pushforward. For this reason, we adapt Definition~\ref{def:sc:frob} of the Frobenius property to pairs of classes of maps. 

\begin{Defn}
    A pair $(J, K)$ of subclasses $J, K \subseteq \C_1$ satisfies the \emph{Frobenius property} if every $f \in K$ does, and likewise the \emph{pushforward property} when every $f \in K$ does. 
\end{Defn}

\section{The Frobenius equivalence for categories of maps}\label{sec:functions}

In this section we move from classes of maps to \emph{categories} of maps and show that the results from the previous section can be generalised to this setting. This generalisation stems from Garner's paper on understanding the small object argument \cite{garner2009UnderstandingSmall} and the extension of the Frobenius equivalence to this setting can be found in \cite{gambino2017FrobeniusCondition}. We give a different account of the Frobenius equivalence here, which shows that it can be derived rather cleanly from the work of Bourke and Garner in \cite{bourke2016AlgebraicWeak}, provided we follow the outline we sketched in Subsection \ref{sec:outlineFrobenius} and take the ``arrow view'' instead of the ``object view'', as discussed in Subsection \ref{sec:sc:frobequiv}; moreover, the extension of the Frobenius equivalence to algebraic weak factorisation systems that we will see in later sections builds on this particular way of developing the Frobenius equivalence for categories of maps.

In what follows we will write $\Ar(\C)$ for the \emph{arrow category} of a category $\C$, that is, the category whose objects are the arrows of $\C$ and whose morphisms are commutative squares in $\C$. Note that $\Ar$ can be considered as a functor $\Cat \to \Cat$, and there are natural transformations $\dom, \cod: \Ar \Rightarrow 1_\Cat$.

We formalise the idea of a category of maps as a functor $\J \to \Ar(\C)$: indeed, we think of $\J$ as consisting of a category of \emph{structured} morphisms in $\C$ and structure-preserving maps between those. The following examples illustrate this.

\begin{Ex}\label{ex:function:splitepi} Recall that a mono $i$ in a category $\C$ is \emph{split} if it has a retraction. One option would be to regard the split monos just as a class of maps, as in the previous section. However, we can also regard the splitting as additional structure, so that we are interested in pairs $(i, r)$ where $r$ is a retraction of $i$. Split monos can then be given categorical structure by defining a morphism between split monomorphisms $(i, r)$, $(i', r')$ as a serially commuting square of functors $(u, v)$:
	\[\begin{tikzcd}
		\bullet & \bullet \\
		\bullet & {\bullet \rlap{ .}}
		\arrow["i"', shift right, from=1-1, to=2-1]
		\arrow["{i'}"', shift right, from=1-2, to=2-2]
		\arrow["r"', shift right, from=2-1, to=1-1]
		\arrow["{r'}"', shift right, from=2-2, to=1-2]
		\arrow["u", from=1-1, to=1-2]
		\arrow["v"', from=2-1, to=2-2]
	\end{tikzcd}\]
		We denote the category of split monos defined in this manner by $\SplMono$, and we have a (faithful) forgetful functor $\SplMono \to \Ar(\C)$ sending the pair $(i,r)$ to $i$.
\end{Ex}

\begin{Ex}\label{ex:function:lali}
    As a special case of the previous example, consider the case where $\C = \Cat$ and the split mono $(i, r)$ is a \emph{split reflection} in that $r \vdash i$ with identity counit (split reflections are better known as right adjoint right inverses or \emph{rari}, for short); see~\cite[Section~4.2]{bourke2016AlgebraicWeak}. We can then define $\SplRef$ as a full subcategory of $\SplMono$ containing the raris in $\Cat$, and we obtain a functor $\SplRef \to \Ar(\Cat)$. Note that a morphism in $\SplMono$ between split reflections is guaranteed to commute with the units of the adjunctions; the reader is referred to~\cite[Section~4.2]{bourke2016AlgebraicWeak} for more details. 
\end{Ex}

\begin{Ex}\label{ex:function:fib}
    A Grothendieck fibration is called split if it comes equipped with an explicit choice of cartesian lifts which preserves identities and compositions. A morphism of split fibrations $P$, $Q$ is defined as a commutative square $(U, V) : P \to Q$ in $\Cat$ that also commutes with the splitting. By this we mean that if $\underline{f}$ is the chosen lift of $f: b \to Pa$ in $\cod P$ and $\underline{Vf}$ the chosen lift of $Vf: Vb \to VPa$ in $\cod Q$, then we  have $\smash{U\underline{f} = \underline{Vf}}$. This yields a category $\SplFib$ together with a (faithful) forgetful functor $\SplFib \to \Ar(\Cat)$. The same definition applies to split opfibrations, so that we also have a forgetful functor $\SplOpFib \to \Ar(\Cat)$.
\end{Ex}

Before we proceed, we make some remarks on notation. In this setting, and in the double categorical setting that follows, we will use the name of the domain $\J$ of a morphism $\J \to \Ar(\C)$ to refer both to the domain and the morphism, rather than naming the morphism separately. Furthermore, following~\cite{bourke2016AlgebraicWeak,bourke2023OrthogonalApproach}, given such a functor $\J \to \Ar(\C)$ we use bold type to denote an element $\f \in J$, and italic type for the corresponding image $f$ of $\f$ under $\J$. Moreover, we use double letters for morphisms in $\J$: that is, if $\f$ and $\g$ are objects in $\J$, then a map $\f \to \g$ will usually be written $\textbf{uv}$ and the image of this map under $\J$ will be $(u,v): f \to g$.

\subsection{Lifting}

Lifting for categories of maps is defined as a more structured version of~(\ref{eq:phi}), demanding a specific and coherent choice of lifts.

\begin{Defn}
    Let $\J, \K \to \Ar(\C)$; a $(\J, \K)$-\emph{lifting operation} $\phi$ assigns to each $\f \in \J$, $\g \in \K$, and lifting problem $(u, v) : f \to g$, a solution $\phi_{\f, \g}(u, v)$:
\begin{equation}\label{eq:horizontalcond}
	\begin{tikzcd}[sep=3em]
		\bullet & \bullet \\
		\bullet & {\bullet \rlap{ .}}
		\arrow["f"', from=1-1, to=2-1]
		\arrow["g", from=1-2, to=2-2]
		\arrow["u", from=1-1, to=1-2]
		\arrow["v"', from=2-1, to=2-2]
		\arrow["{\phi_{\f, \g}(u, v)}"{description}, dotted, from=2-1, to=1-2]
	\end{tikzcd}
\end{equation}
	This choice is required to respect morphisms of $\J$ and $\K$, in the sense that if $\textbf{wx}: \f' \to \f$ in $\J$, $(u, v) : f \to g$, and $\textbf{yz} : \g \to \g'$ in $\K$, then $y.\phi_{\f, \g}(u, v).x = \phi_{\f', \g'}(w.u.y, x.v.z)$:
\[\begin{tikzcd}[sep=small]
	\bullet && \bullet && \bullet && \bullet && \bullet && \bullet \\
	&&&&&&& {=} \\
	\bullet && \bullet && \bullet && \bullet && \bullet && {\bullet \rlap{ .}}
	\arrow["u", from=1-3, to=1-5]
	\arrow["v"', from=3-3, to=3-5]
	\arrow["g", from=1-5, to=3-5]
	\arrow["f"', from=1-3, to=3-3]
	\arrow["{f'}"', from=1-1, to=3-1]
	\arrow["{g'}", from=1-7, to=3-7]
	\arrow["x"', from=3-1, to=3-3]
	\arrow["y", from=1-5, to=1-7]
	\arrow["{\smash{\phi_{\f, \g}}}"{description}, dotted, from=3-3, to=1-5]
	\arrow["w", from=1-1, to=1-3]
	\arrow["z"', from=3-5, to=3-7]
	\arrow["{f'}"', from=1-9, to=3-9]
	\arrow["{g'}", from=1-11, to=3-11]
	\arrow["{y.u.w}", from=1-9, to=1-11]
	\arrow["{z.v.x}"', from=3-9, to=3-11]
	\arrow["{\smash{\phi_{\f', \g'}}}"{description}, dotted, from=3-9, to=1-11]
\end{tikzcd}\]
\end{Defn}

\begin{Rk}
	Note that the lifting operation chooses the lift on the basis of $\f$ and $\g$ rather than $f$ and $g$. This means, for example, that if $\K$ is the category $\SplMono$ of split monos, as in Example \ref{ex:function:splitepi}, the choice of lift may depend on the retraction.
\end{Rk}

The assignment $(\J, \K) \mapsto \CLift(\J, \K)$ of functors $\J, \K$ to the class $\CLift(\J, \K)$ of $(\J, \K)$-lifting operations can be made the object part of a functor 
\begin{equation}\label{eq:funcliftfunctor}
    \CLift : (\Cat/\Ar(\C))\op \times (\Cat/\Ar(\C))\op \to \Set.
\end{equation}
A $(\J,\K)$-lifting operation is to functors $\J, \K \to \Ar(\C)$ what the statement $J \pitchfork K$ is to subclasses $J, K \subseteq \C_1$. Further extending this analogy, we have the following restriction of~\cite[Proposition 15]{bourke2016AlgebraicWeak} to the case of categories of maps. 
\begin{Prop}\label{prop:functor:galoisconnection}
    $\CLift$ is representable in both arguments, inducing an adjunction:
    \begin{equation}\label{eq:functor:galoisconnection}
\begin{tikzcd}
	{(\Cat/\Ar(\C))\op} & {\Cat/\Ar(\C) \rlap{ .}}
	\arrow[""{name=0, anchor=center, inner sep=0}, "{\lpf(-)}"', shift right=2, from=1-2, to=1-1]
	\arrow[""{name=1, anchor=center, inner sep=0}, "{(-)\lpf}" {below, overlay}, shift right=2, from=1-1, to=1-2]
	\arrow["\dashv"{anchor=center, rotate=-90}, draw=none, from=0, to=1]
\end{tikzcd}
\end{equation}
\end{Prop}
\begin{proof}
The representing object of $\CLift(-, \J)$ is a category $\lpf \J \to \Ar(\C)$ of which the objects are pairs $(f, \phi_{f-})$ with $\phi_{f-} \in \CLift(f, \J)$ (where we consider $f$ to be a functor $f: 1 \to \Ar(\C)$), and in which a morphism $(f, \phi_{f-}) \to (g, \phi_{g-})$ is a square $(u, v) : f \to g$ which commutes with the lifting operations. This means that for $\j \in \J$ and $(w, x) : g \to j$ we have $\phi_{f, \j}(w.u, x.v) = \phi_{g, \j}(w, x).v$:
\[\begin{tikzcd}[sep=small]
	\bullet && \bullet && \bullet && \bullet && \bullet \\
	&&&&& {=} \\
	\bullet && \bullet && \bullet && \bullet && {\bullet \rlap{ .}}
	\arrow["f"', from=1-1, to=3-1]
	\arrow["g"', from=1-3, to=3-3]
	\arrow["j", from=1-5, to=3-5]
	\arrow["u", from=1-1, to=1-3]
	\arrow["v"', from=3-1, to=3-3]
	\arrow["w", from=1-3, to=1-5]
	\arrow["x"', from=3-3, to=3-5]
	\arrow["f"', from=1-7, to=3-7]
	\arrow["j", from=1-9, to=3-9]
	\arrow["{w.u}", from=1-7, to=1-9]
	\arrow["{x.v}"', from=3-7, to=3-9]
	\arrow["{\smash{\phi_{g, \j}}}"{description}, dotted, from=3-3, to=1-5]
	\arrow["{\smash{\phi_{f, \j}}}"{description}, dotted, from=3-7, to=1-9]
\end{tikzcd}\]
The representing object $\J\lpf \to \Ar(\C)$ of $\CLift(\J, -)$ is defined analogously. 
\end{proof}

Birepresentability of (\ref{eq:funcliftfunctor}) gives us the following analog of~(\ref{eq:ssJpfK}):
\begin{equation}\label{eq:funcJpfK}
    \Cat/\Ar(C)(\J, \lpf \K) \cong \CLift(\J, \K) \cong \Cat/\Ar(\C)(\K, \J \lpf).
\end{equation}
In turn, this gives an analog of Definition~\ref{def:closedliftingpair} of (closed) lifting pairs. However, since this is no longer a property but additional structure---in the form of a $(\J, \K)$-lifting operation $\phi$---we will now speak of \emph{lifting structures}, following terminology of~\cite{bourke2023OrthogonalApproach}.

\begin{Defn}\label{def:funcliftingstructure}
    A triple $(\J, \phi, \K)$ of functors $\J, \K \to \Ar(\C)$ and a $(\J, \K)$-lifting operation $\phi$ is called a \emph{lifting structure}. A lifting structure $(\J, \phi, \K)$ is called \emph{closed} when the functors $\phi_l : \J \to \lpf \K$ and $\phi_r : \K \to \J \lpf$ induced by (\ref{eq:funcJpfK}) are invertible. 
\end{Defn} 
\begin{Ex}
    Let $\J \to \Ar(\C)$; applying~(\ref{eq:funcJpfK}) to the identities on $\lpf \J$ and $\J \lpf$ yields lifting operations $\lcan \in \CLift(\lpf \J, \J)$ and $\rcan \in \CLift(\J, \J\lpf)$, and so we get, as we did in Example~\ref{ex:sscanliftingpairs}, two canonical lifting structures $(\lpf \J, \lcan, \J)$ and $(\J, \rcan, \J\lpf)$. Again, this yields lifting structures $(\lpf \J, \rcan, (\lpf \J)\lpf)$ and $(\lpf(\J \lpf), \lcan, \J\lpf)$ which are said to be fibrantly generated and cofibrantly generated by $\J$. Whether these are closed in general seems to be an open question; see~\mbox{\cite[Remark 10]{bourke2023OrthogonalApproach}}.
\end{Ex}

\begin{Ex} \label{ex:liftingstrforrari}
	The functors $\SplRef \to \Ar(\Cat)$ and $\SplFib \to \Ar(\Cat)$ form part of a lifting structure. Indeed, suppose
	\begin{displaymath}
		\begin{tikzcd}
			\ct{C} \ar[d, "R"'] \ar[r, "X"] & \ct{A} \ar[d, "P"] \\
			\ct{D} \ar[r, "Y"] \ar[ur, dotted, "\phi"] & \ct{B}
		\end{tikzcd}
	\end{displaymath}
	is a commutative diagram of categories in which $P$ is a split fibration and $R$ is a split reflection with left adjoint $L \ladj R$ and unit $\theta: 1 \to R.L$. If $d$ is an object in \ct{D}, then we can consider the image of the unit $\theta_d: d \to RLd$ under $Y$. Since $YRLd = PXLd$ and $P$ is split fibration, we find a cartesian lift of $P\theta_d$ as in
	\begin{equation}\label{eq:liftingforraris}
		\begin{tikzcd}
		\phi(d) & XLd \\
		Yd & {YRLd \rlap{ ,}}
		\arrow["P\theta_d", from=2-1, to=2-2]
		\arrow["{\underline{P\theta_d}}", from=1-1, to=1-2]
		\arrow[dashed, no head, from=1-1, to=2-1]
		\arrow[dashed, no head, from=1-2, to=2-2]
	\end{tikzcd}
	\end{equation}
	whose domain we define to be $\phi(d)$. This definition of $\phi$ can be extended to morphisms $d \to d'$ using that $\underline{P\theta_{d'}}$ is cartesian. Then, $P.\phi = Y$ by construction, while $\phi. R = X$ follows from $\theta_{Rc} = 1_{Rc}$ and the fact that the splitting of $P$ lifts identities to identities. It is readily verified that this way of constructing the lift $\phi$ respects the morphisms in both $\SplRef$ and $\SplFib$; for more details, the reader is referred to \cite[Examples 4 (ii)]{bourke2023OrthogonalApproach}. Note that this lifting structure is not closed, as can be seen from the fact that split fibration are not closed under retracts.
\end{Ex}

\subsection{Base change}

As with classes of maps, lifting structures can be phrased relative to adjunctions of categories. Following the approach in~\cite[Section~6.4]{bourke2016AlgebraicWeak}, we do so by extending the adjunction~(\mbox{\ref{eq:functor:galoisconnection}}). Let $\Cat_{\text{ladj}}$ denote the category of small categories with fully specified adjunctions between them, pointing in the direction of the left adjoint; $\Cat_{\text{radj}}$ is defined dually, i.e.\ $\Cat_{\text{radj}} = \Cat_{\text{ladj}}\op$. These come with functors $\Catladj, \Catradj \to \Cat$ forgetting the right and left adjoints, respectively. We now define $\cladj$ by the following pullback: 
\[\begin{tikzcd}[row sep=scriptsize]
	\cladj & {\Cat_\tladj} \\
	{\Ar(\Cat)} & {\Cat \rlap{ .}}
	\arrow["\cod", from=2-1, to=2-2]
	\arrow[from=1-1, to=2-1]
	\arrow[from=1-1, to=1-2]
	\arrow["\lrcorner"{anchor=center, pos=0.125}, draw=none, from=1-1, to=2-2]
	\arrow[from=1-2, to=2-2]
\end{tikzcd}\]
The category $\cradj$ is defined similarly, using $\Cat_\tradj \to \Cat$.
\begin{Prop}\label{prop:sc:bcgaloisconnection}
    The adjunction of Proposition~\ref{prop:functor:galoisconnection} extends to:
    \begin{equation}\label{eq:function:bcgaloisconnection}
	\begin{tikzcd}
		{(\cladj)\op} & {\cradj \rlap{ .}}
		\arrow[""{name=0, anchor=center, inner sep=0}, "{\lpf(-)}"', shift right=2, from=1-2, to=1-1]
		\arrow[""{name=1, anchor=center, inner sep=0}, "{(-)\lpf}" {below, overlay}, shift right=2, from=1-1, to=1-2]
		\arrow["\dashv"{anchor=center, rotate=-90}, draw=none, from=0, to=1]
\end{tikzcd}
    \end{equation}
\end{Prop}
\begin{proof}
	The functors $\lpf(-)$ and $(-)\lpf$ act as before on objects, and on arrows by the construction in Proposition~\ref{prop:sc:basechange}. See~\cite[Proposition 21]{bourke2016AlgebraicWeak} for further details.
\end{proof}

The Frobenius equivalence relies essentially on the fact that~(\ref{eq:function:bcgaloisconnection}) is an adjunction, so it is worth stating this property separately; it is the categorical analog of Proposition~\ref{prop:sc:basechange}. 

\begin{Cor}\label{cor:function:basechange}
    Let $F \dashv G : \D \to \C$ be an adjunction, and $\J \to \Ar(\C)$, $\K \to \Ar(\D)$ be functors; then there is a bijection between lifts $\F$ and $\G$ as in the diagrams below:
\[\begin{tikzcd}[ampersand replacement=\&]
	\J \& {\lpf \K} \& \K \& {\J \lpf} \\
	{\Ar(\C)} \& {\Ar(\D) \rlap{ ,}} \& {\Ar(\D)} \& {\Ar(\C) \rlap{ .}}
	\arrow["\F", dotted, from=1-1, to=1-2]
	\arrow[from=1-2, to=2-2]
	\arrow[from=1-1, to=2-1]
	\arrow["{\Ar(F)}", from=2-1, to=2-2]
	\arrow["\G", dotted, from=1-3, to=1-4]
	\arrow[from=1-3, to=2-3]
	\arrow[from=1-4, to=2-4]
	\arrow["{\Ar(G)}", from=2-3, to=2-4]
\end{tikzcd}\]
\end{Cor}

\subsection{Slicing}\label{sec:function:slicing}

We now proceed to define an analog of slicing for functors $\J \to \Ar(\C)$. 

\begin{Defn}\label{def:funcslicing}
    Given a functor $\J \to \Ar(\C)$ and an object $A \in \C$ we define the \emph{slice} functor $\J/A \to \Ar(\C/A)$ as the pullback of $\J$ along $\Ar(\dom) : \Ar(\C/A) \to \Ar(\C)$: 
\[\begin{tikzcd}
	{\J/A} & \J \\
	{\Ar(\C/A)} & {\Ar(\C) \rlap{ .}}
	\arrow[from=1-2, to=2-2]
	\arrow["{\Ar(\dom)}", from=2-1, to=2-2]
	\arrow[from=1-1, to=2-1]
	\arrow[from=1-1, to=1-2]
	\arrow["\lrcorner"{anchor=center, pos=0.125}, draw=none, from=1-1, to=2-2]
\end{tikzcd}\]
Here we denote by $\dom$ the forgetful functor $\C/A \to \C : a \mapsto \dom a$. 
\end{Defn}

Spelling out the meaning of Definition~\ref{def:funcslicing}, we find that an object of $\J/A$ is equivalently a pair $(\f, a)$ with $\f \in \J$ and $(f, a)$ a morphism in $\C/A$. The associated functor to $\Ar(\C/A)$ forgets this $\J$-structure, i.e.\ $\J/A \to \Ar(\C/A): (\f, a) \mapsto (f, a)$. 

The goal will now be to prove that this slicing operation preserves closed lifting structures. We begin by adapting the construction of Lemma~\ref{lem:slicelift}.
\begin{Prop}\label{prop:function:slicels}
    Let $(\J, \phi, \K)$ be a lifting structure on $\C$ and $A \in \C$ some object, then there is a lifting structure $(\J/A, \phi/A, \K/A)$ on $\C/A$. 
\end{Prop}
\begin{proof}
	We define the $(\J/A, \K/A)$-lifting operation $\phi/A$ by:
    \[\phi/A_{(\f, a), (\g, b)}((u, c), (v, d)) = (\phi_{\f, \g}(u, v), b.g). \qedhere\]
\end{proof}
\begin{Prop}\label{prop:function:slicecls}
    If $(\J, \phi, \K)$ is a closed lifting structure then so is $(\J/A, \phi/A, \K/A)$.
\end{Prop}

To prove Proposition~\ref{prop:function:slicecls}, we use the strategy from Section~\ref{sec:ssslicing} in showing that the operations defined in Proposition~\ref{prop:functor:galoisconnection} commute with the slicing operation, which is to say that $(\J\lpf)/A \cong (\J/A)\lpf$ and $(\lpf \J)/A \cong \lpf(\J/A)$ as functors over $\Ar(\C/A)$. Again, only the first of these isomorphisms exists in general. For the second desired isomorphism we do always obtain a functor $(\lpf \J)/A \to \lpf(\J/A)$, but for its inverse to exist we require that $\J$, in some suitable sense, is closed under pullbacks. 
\begin{Prop}\label{prop:function:sc}
    Given a functor $\J \to \Ar(\C)$ and an object $A \in \C$ there are functors $(\lambda/A)_l : (\lpf \J)/A \to \lpf(\J/A)$ and $(\rho / A)_r : (\J\lpf)/A \to (\J/A)\lpf$.
\end{Prop}
\begin{proof}
	Apply Proposition~\ref{prop:function:slicels} to the canonical lifting structures induced by $\J$.
\end{proof}
\begin{Prop}\label{prop:function:commrl}
 The functor $(\rho / A)_r$ of Proposition~\ref{prop:function:sc} is an isomorphism. 
\end{Prop} 
\begin{proof}
We make a functor $(\J/A)\lpf \to \J\lpf$ using the construction in Lemma~\ref{lem:inclslicingcomrl}; given $(f,a)$ with a $(\J/A, (f, a))$-lifting operation $\smash{\phi_{-(f,a)}}$ we define a $(\J, f)$-lifting operation $\theta_{-f}$ by $\theta_{\j, f}(u, v) = \phi_{(\j, a.v), (f, a)}((u, a.f), (v, a))$. By the universal property of $(\J\lpf)/A$ this induces the desired inverse $(\rho / A)_r^{-1} : (\J/A)\lpf \to (\J\lpf)/A$. 
\end{proof}

\begin{Defn}\label{defn:function:comprcat}
    A functor $\J \to \Ar(\C)$ is a \emph{comprehension category} if $\cod.\J$ is a Grothendieck fibration making $\J$ into a cartesian functor:
	\begin{displaymath}
		\begin{tikzcd}
			\J \ar[rr, "\J"] \ar[dr, "{\cod.\J}"'] & & \Ar(\C) \ar[dl, "{\cod}"] \\
			& \C \rlap{ .}
		\end{tikzcd}
	\end{displaymath}
\end{Defn}
\begin{Prop}
    Let $\J \to \Ar(\C)$; then $\J\lpf \to \Ar(\C)$ is a comprehension category.
\end{Prop}
\begin{proof}
    Assume we have maps $g: B \to A$ and $y: C \to A$ and $\phi_{-g}$ is a $(\J, g)$-lifting operation. Since \ct{C} has pullbacks, there is a pullback square $(x, y) : g' \to g$, inducing a $(\J, g')$-lifting operation $\psi_{-g'}$ for $g'$ by the construction of Lemma~\ref{lem:sc:pb}. Now $(x,y): (g', \psi_{g'-}) \to (g, \phi_{g-})$ is a cartesian morphism in $\J\lpf$.
\end{proof}
The following is~\cite[Proposition~5.4]{gambino2017FrobeniusCondition}. 
\begin{Prop}\label{prop:function:commll}
    Let $\J \to \Ar(\C)$ be a comprehension category, and $A \in \C$ be some object; then the functor $(\lambda / A)_l$ of Proposition~\ref{prop:function:sc} is an isomorphism.
\end{Prop}
\begin{proof}
	Let $((g, a), \phi_{(g, a)-}) \in {^\pf}(\J/A)$ and consider $\f \in \J$ with $(u, v) : g \to f$. By assumption there is a cartesian morphism $v^*\f \to \f \in \J$ over $v$ and a pullback square $(v^+, v) : v^*f \to f$, and so a lift $\phi_{(g, a), (v^*\f, a)}(\alpha, 1)$ of the induced morphism $\alpha$:
	\begin{equation}\label{eq:liftop}
\begin{tikzcd}[ampersand replacement=\&]
	\bullet \& \bullet \& \bullet \\
	\bullet \& \bullet \& {\bullet \rlap{ .}}
	\arrow["g"', from=1-1, to=2-1]
	\arrow["f", from=1-3, to=2-3]
	\arrow["1"', from=2-1, to=2-2]
	\arrow["v"', from=2-2, to=2-3]
	\arrow["{v^+}", from=1-2, to=1-3]
	\arrow["v^*f", from=1-2, to=2-2]
	\arrow["\alpha", from=1-1, to=1-2]
	\arrow["\phi", dotted, from=2-1, to=1-2]
	\arrow["\lrcorner"{anchor=center, pos=0.125}, draw=none, from=1-2, to=2-3]
\end{tikzcd}
	\end{equation}
	In other words, following the construction in Lemma~\ref{lem:sc:pb}, we define a $(g, \J)$-lifting operation $\psi_{g-}$ on $(u, v) : g \to f$ for $\f\in \J$ by
	\begin{equation}\label{eq:sliceliftop}
		\psi_{g, \f}(u, v) = v^+.\phi_{(g, a), (v^*\f, a)}(\alpha, 1).
	\end{equation}
	 This morphism is independent of the choice of cartesian lift $v^*\f \to \f$, by an argument similar to the one which now follows.
	
	We should check that the lifting function $\psi_{g-}$ defined this way satisfies the functoriality condition of~(\ref{eq:horizontalcond}). To this end, consider $\mathbf{w}\mathbf{x} : \f \to \f'$ in $\J$ and $(u, v) : g \to f$. We get lifts $\phi(\alpha) = \phi_{(g, a), (v^*\f, a)}(\alpha,1)$ and $\phi(\beta) = \phi_{(g,a), (v^*\f',1)}(\beta,1)$, as on the left and middle below, with $y = x.v$, $v^*\f \to \f$ cartesian over $v$, and $v^*\f' \to \f'$ cartesian over $y$:
	\[\begin{tikzcd}[ampersand replacement=\&]
		\bullet \& \bullet \& \bullet \& \bullet \& \bullet \& \bullet \& \bullet \& \bullet \& \bullet \& \bullet \\
		\bullet \& \bullet \& {\bullet } \& {\bullet \rlap{ ,}} \& \bullet \& \bullet \& {\bullet \rlap{ ,}} \& \bullet \& \bullet \& {\bullet\rlap{ .}}
		\arrow["g"', from=1-1, to=2-1]
		\arrow["v^*f", from=1-2, to=2-2]
		\arrow["f", from=1-3, to=2-3]
		\arrow["v"', from=2-2, to=2-3]
		\arrow["1"', from=2-1, to=2-2]
		\arrow["\alpha", from=1-1, to=1-2]
		\arrow["{v^+}", from=1-2, to=1-3]
		\arrow["{\phi(\alpha)}"{description}, from=2-1, to=1-2]
		\arrow["g"', from=1-5, to=2-5]
		\arrow["{v^*f'}", from=1-6, to=2-6]
		\arrow["{f'}", from=1-7, to=2-7]
		\arrow["{y^+}", from=1-6, to=1-7]
		\arrow["y"', from=2-6, to=2-7]
		\arrow["\beta", from=1-5, to=1-6]
		\arrow["1"', from=2-5, to=2-6]
		\arrow["{\phi(\beta)}"{description}, from=2-5, to=1-6]
		\arrow["g"', from=1-8, to=2-8]
		\arrow["\alpha", from=1-8, to=1-9]
		\arrow["1"', from=2-8, to=2-9]
		\arrow["\gamma", from=1-9, to=1-10]
		\arrow["1"', from=2-9, to=2-10]
		\arrow["{v^*f'}", from=1-10, to=2-10]
		\arrow["v^*f"', from=1-9, to=2-9]
		\arrow["{\phi(\alpha)}"{outer sep=-2pt, pos=0.5}, from=2-8, to=1-9]
		\arrow["{\phi(\beta)}"'{outer sep=-2pt, pos=0.7}, curve={height=5pt}, from=2-8, to=1-10]
		\arrow["{f'}", from=1-4, to=2-4]
		\arrow["w", from=1-3, to=1-4]
		\arrow["x"', from=2-3, to=2-4]
	\end{tikzcd}\]
	It follows from the universal property of the cartesian lift $v^*\f' \to \f'$ that there is a $\J$-morphism $\bm{\gamma} : v^*\f \to v^*\f'$ over $(\gamma,1)$, as on the right above, and so we get $\gamma.\phi(\alpha) = \phi(\gamma.\alpha) = \phi(\beta)$; hence $w.v^+.\phi(\alpha) = y^+.\gamma.\phi(\alpha) = y^+.\phi(\beta)$ as desired.
\end{proof}

\subsection{Frobenius equivalence}\label{sec:function:frobenius}

It is now straightforward to define and prove a Frobenius equivalence for categories of maps.

\begin{Defn}\label{def:function:frobstruc}
    Let $f: A \to B$ be a map in $\C$, and $(\J, \K)$ a pair of functors over $\Ar(\C)$. A \emph{Frobenius structure} $\f^*$ on $f$ for $(\J, \K)$ is a filler as on the left below: 
\[\begin{tikzcd}[ampersand replacement=\&]
	{\J/B} \& {\J/A} \& {\K/A} \& {\K/B} \\
	{\Ar(\C/B)} \& {\Ar(\C/A) \rlap{ ,}} \& {\Ar(\C/A)} \& {\Ar(\C/B) \rlap{ .}}
	\arrow["{\Ar(f^*)}", from=2-1, to=2-2]
	\arrow[from=1-1, to=2-1]
	\arrow[from=1-2, to=2-2]
	\arrow["{\f^*}", dotted, from=1-1, to=1-2]
	\arrow["{\Ar(f_*)}", from=2-3, to=2-4]
	\arrow[from=1-3, to=2-3]
	\arrow[from=1-4, to=2-4]
	\arrow["{\f_*}", dotted, from=1-3, to=1-4]
\end{tikzcd}\]
Similarly, a \emph{pushforward structure} $\f_*$ on $f$ is a filler as on the right above. 
\end{Defn}
\begin{Thm}
	For a map $f : A \to B$ in $\C$ and a closed lifting structure $(\J, \phi, \K)$ on $C$, there is a bijection between Frobenius and pushforward structures on $f$ for $(\J, \K)$.
\end{Thm}
\begin{proof}
	The previous results give us the following string of isomorphisms:
	\begin{align*}
		\cladj(\J/B, \J/A) &\cong \cladj(\J/B, \lpf(\K/A)) \\
		&\cong \cradj(\K/A, (\J/B)\lpf) \\
		&\cong \cradj(\K/A, \K/B),
	\end{align*}
	where we consider morphisms in $\cladj$ and $\cradj$ w.r.t.\ $f^* \ladj f_*$.
\end{proof}

\begin{Defn}\label{def:function:JKfrobstruc}
    A \emph{Frobenius structure} for a pair $(\J, \K)$ of functors over $\Ar(C)$ is a choice of Frobenius structure $\f^*$ for every $\f \in K$; likewise, a \emph{pushforward structure} for $(\J, \K)$ is a choice of pushforward structure $\f_*$ for every $\f \in K$. 
\end{Defn}

\subsection{An example of split reflections and split opfibrations}

We now give an example of a Frobenius structure in the sense of Definition~\ref{def:function:JKfrobstruc}, building on~\cite[Proposition~5.3]{gambino2023ModelsMartinLof}. Recall from Examples~\ref{ex:function:lali}~and~\ref{ex:function:fib} the functors $\SplRef \to \Ar(\Cat)$ and $\SplOpFib \to \Ar(\Cat)$. Given a split opfibration $P : \A \to \B$, an object $a \in \A$ and a morphism $Pa \to b$ in $\B$ we write $\smash{\underline{f} : a \to f_!a}$ for the choice of cocartesian lift of $f$ along $a$:
\[\begin{tikzcd}
	a & {f_!a} \\
	Pa & {b \rlap{ .}}
	\arrow["f", from=2-1, to=2-2]
	\arrow["{\underline{f}}", from=1-1, to=1-2]
	\arrow[dashed, no head, from=1-1, to=2-1]
	\arrow[dashed, no head, from=1-2, to=2-2]
\end{tikzcd}\] 
Note that this notation is imprecise, since the lift $\smash{\underline{f}}$ also depends on $a$. Each morphism $f: b \to b' \in \B$ thus induces a functor $f_! : \A(b) \to \A(b')$, where $\A(b)$ denotes the preimage category under $P$, sending an element $a \in \A(b)$ to the codomain of $\smash{\underline{f}}$. 

\begin{Prop}\label{prop:functor:frobstruc}
	The pair $(\SplRef, \SplOpFib)$ admits a Frobenius structure. 
\end{Prop}
\begin{proof}
	We construct a lift $\P^* : \SplRef/\B \to \SplRef/\A$ for an opfibration $P : \A \to \B$:
\[\begin{tikzcd}
	{\SplRef/\B} & {\SplRef/\A} \\
	{\Ar(\Cat/\B)} & {\Ar(\Cat/\A) \rlap{ .}}
	\arrow["{\Ar(P^*)}", from=2-1, to=2-2]
	\arrow[from=1-2, to=2-2]
	\arrow[from=1-1, to=2-1]
	\arrow["{\P^*}", from=1-1, to=1-2]
\end{tikzcd}\]
Consider a situation as drawn on the left below, where $L \dashv R$ is a split reflection with unit $\eta$; the goal is to construct a split reflection $G \dashv P^*R$ as on the left of:
\begin{equation}\label{eq:laliconstr}
	\begin{tikzcd}
		{\A\times_\B\D} & \D \\
		{\A\times_\B\C} & \C \\
		\A & {\B \rlap{ ,}}
		\arrow["P"', from=3-1, to=3-2]
		\arrow["U", from=2-2, to=3-2]
		\arrow["R", shift left, from=1-2, to=2-2]
		\arrow["{P^*U}", shift left, from=2-1, to=3-1]
		\arrow["F"', dotted, shift right, from=2-1, to=3-1]
		\arrow["{P^*R}", shift left, from=1-1, to=2-1]
		\arrow["G", shift left, dotted, from=2-1, to=1-1]
		\arrow["L", shift left, from=2-2, to=1-2]
		\arrow[from=1-1, to=1-2]
		\arrow["{\pi_\C}"', from=2-1, to=2-2]
	\end{tikzcd}
	\qquad\qquad\quad
	\begin{tikzcd}
		a & {(U\eta_c)_!a} \\
		Uc & {URLc \rlap{ .}}
		\arrow["{U\eta_c}", from=2-1, to=2-2]
		\arrow[dashed, no head, from=1-1, to=2-1]
		\arrow["\underline{U\eta_c}", from=1-1, to=1-2]
		\arrow[dashed, no head, from=1-2, to=2-2]
	\end{tikzcd}
\end{equation}
To achieve this we first define an auxiliary functor $F : \A \times_\B \C \to \A : (a, c) \mapsto (U\eta_c)_! a$, as illustrated on the right above. The action of $F$ on a morphism $(a_1, c_1) \to (a_2, c_2)$ is defined using the universal property of $\smash{\underline{U\eta_{\smash{c_1}}}}$. Now $G$ is defined by $(a, c) \mapsto (F(a, c), Lc)$, i.e.\ $G = F \times_\B L.\pi_\C$. Lastly, we need a unit $\theta : 1 \To P^*R.G$, for which we take \[\theta_{(a,c)} = (\underline{U\eta_c},\eta_c) : (a, c) \to (F(a, c), RLc).\]

We have to show that $G$ is a left inverse of $P^*R$, and that the triangle identities $\theta.P^*R = 1$ and $G.\theta = 1$ hold. To start, we note that because $\eta.R = 1$, and the splitting of $P$ preserves units, we have for $(a, d) \in \A \times_\B \D$ that 
\begin{equation}\label{eq:star}
\underline{U\eta_{Rd}} = \underline{U1_{Rd}} = \underline{1_{URd}} = 1_a.
\end{equation}
Therefore, $F(a, Rd) = (U\eta_{Rd})_!a = (1_a)_!a = a$, and so \[G(P^*R(a,d)) = G(a, Rd) = (F(a, Rd), LRd) = (a, d),\] which is to say that $G$ is a left inverse of $P^*R$. Similarly, from~(\ref{eq:star}) we get \[\theta_{P^*R(a,d)} = \theta_{(a, Rd)} = (\smash{\underline{U\eta_{Rd}}}, \eta_{Rd}) = (1_a, 1_{Rd}).\] Lastly, we check if $G\theta_{(a,c)} = (1_{F(a, c)}, 1_{Lc})$. As $L.\eta = 1$ by assumption, this comes down to verifying that $F(\smash{\underline{U\eta_c}}, \eta_c) = 1_{F(a, c)} : F(a, c) \to F(F(a, c), RLc)$. Substituting $Lc$ for $d$ in~(\ref{eq:star}) we find that $\smash{\underline{U\eta_{RLc}}} = 1_{F(a, c)}$ and so $F(F(a, c), RLc) = F(a,c)$. Therefore, the desired result follows from the universal property of the lift $\smash{\underline{U\eta_c}}$. 

It remains to check that $\P^*$ preserves morphisms of split reflections; to this end, we consider a morphism of split reflections $(X, Y) : (L', R') \to (L, R)$ as drawn below:
\[\begin{tikzcd}[column sep=small]
	{\A\times_\B\mathcal{F}} &[-0.2cm]&[-0.2cm] {\A\times_\B\D} && {\mathcal{F}} && \D \\
	{\A\times_\B\E} && {\A\times_\B\C} && \E && \C \\
	& \A &&&& {\B \rlap{ .}}
	\arrow["P", from=3-2, to=3-6]
	\arrow["R", shift left, from=1-7, to=2-7]
	\arrow["L", shift left, from=2-7, to=1-7]
	\arrow["{R'}", shift left, from=1-5, to=2-5]
	\arrow["{L'}", shift left, from=2-5, to=1-5]
	\arrow["X", from=1-5, to=1-7]
	\arrow["Y", from=2-5, to=2-7]
	\arrow["U", from=2-7, to=3-6]
	\arrow["{U'}"', from=2-5, to=3-6]
	\arrow["{P^*Y}", from=2-1, to=2-3]
	\arrow["{P^*X}", from=1-1, to=1-3]
	\arrow["{P^*R'}", shift left, from=1-1, to=2-1]
	\arrow["{G'}", shift left, from=2-1, to=1-1]
	\arrow["G", shift left, from=2-3, to=1-3]
	\arrow["{P^*R}", shift left, from=1-3, to=2-3]
	\arrow["{P^*U'}"', from=2-1, to=3-2]
	\arrow["{P^*U}", from=2-3, to=3-2]
\end{tikzcd}\]
We want that $G.P^*Y = P^*X.G'$, i.e.\ that $(F(a, Ye), LYe) = (F'(a, e), XL'e)$ for $(a, e) \in \A\times_\B\E$. In fact, both equalities are simple consequences of the assumption that $L.Y = X.L'$ (which implies $\eta \circ Y = Y \circ \eta'$). 
\end{proof}

\section{Algebraic Weak Factorisation Systems}

An \awfs\ is a significantly more structured object than a \wfs, and at first sight their definitions appear quite distinct, because \awfs\ are usually defined in terms of an interacting monad-comonad pair, making no explicit mention of liftings as in diagram~(\ref{eq:phi}). However, it was recently shown in~\cite{bourke2023OrthogonalApproach} that \awfs\ admit an equivalent definition---there called \emph{lifting} \awfs\ for distinction---very much akin to Definition~\ref{def:wfs} of \wfs. In this definition of \awfs, the role of the subclasses $L, R \subseteq \C_1$ in the definition of a \wfs\ are played by \emph{double functors} $\bL, \bR \to \Sq(\C)$ over the double category $\Sq(\C)$ of squares in $\C$. To explain what that means, let us first briefly recall what we mean by double categories.

\subsection{Double categories}

A double category $\bC$ can be succinctly defined as an internal category in $\Cat$. This amounts to the following data, satisfying the usual axioms of a category:
\[\begin{tikzcd}
	{\bC_0} & {\bC_1} & {\bC_1 \times_{\bC_0} \bC_1 \rlap{ .}}
	\arrow["c"', from=1-3, to=1-2]
	\arrow["{\text{id}}"{description}, from=1-1, to=1-2]
	\arrow["s"', shift right=2, from=1-2, to=1-1]
	\arrow["t", shift left=2, from=1-2, to=1-1]
\end{tikzcd}\]
More specifically, a double category $\bC$ consists of a category $\bC_0$ of objects and a category $\bC_1$ of arrows, with operations $\text{id}, s, t, c$ for identity arrow assignment, source, target, and composition operations, respectively. A double category has two types of morphisms: there are the morphisms of $\bC_0$, called the \emph{horizontal} morphisms of $\bC$; and the objects of $\bC_1$, called the \emph{vertical} morphisms of $\bC$. The morphisms of $\bC_1$ are called \emph{squares} of $\bC$. There is an associated notion of double functor, which yields the category $\Dbl$ of double categories. There is a functor $\Sq : \Cat \to \Dbl$ sending a category $\C$ to its \emph{double category of squares} $\Sq(\C)$. The object category of $\Sq(\C)$ is $\C$ itself, and the arrow category of $\Sq(\C)$ is the arrow category $\Ar(\C)$. 

We saw in Section~\ref{sec:functions} that a notion of structured morphisms in the ambient category $\C$ can be given as a functor $\J \to \Ar(\C)$. If we wish to impose in addition a condition saying that the morphisms in $\J$ are closed under composition, the natural way to capture this is as a double functor over $\Sq(\C)$, i.e.\ morphisms $\DJ \to \Sq(\C)$ in $\Dbl$. Note, however, that in principle a double functor $\DJ \to \Sq(\C)$ endows $\C$ with more structure than just that on its arrows; for example, it also imposes structures on its objects through the functor $\DJ_0 \to \C$. In many examples of interest, no additional structure is imposed except that on the arrows. Such double categories are called concrete in~\cite{bourke2016AlgebraicWeak}. 
\begin{Defn}
	A double functor $\DJ \to \Sq(\C)$ is \emph{concrete} if $\DJ_0$ is an isomorphism and $\DJ_1$ is faithful.  
\end{Defn}

By internalising the definition of natural transformations the category $\Dbl$ can be made a 2-category. Spelling this out, a \emph{double} natural transformation $\alpha : F \To G$ between double functors $F, G : \DC \to \DD$ is a functor $\alpha : \DC_0 \to \DD_1$ subject to various conditions, such as $s_\DD.\alpha = F_0$. This means that, by virtue of how we defined horizontal and vertical morphisms, a double natural transformation would assign to each object of $\DC$ a \emph{vertical} morphism of $\DD$. However, a second candidate for the 2-cells of $\Dbl$ is given by first \emph{transposing} the double categories and functors involved. This is a type of duality for double categories in which the sets of vertical and horizontal arrows are switched. We thus obtain a notion of double natural transformation which assigns to each object of $\DC$ a \emph{horizontal} morphism of $\DD$, again subject to various conditions. Following~\cite{bourke2016AlgebraicWeak}, we will consider $\Dbl$ as a 2-category by taking the latter of these two options for the 2-cells. 

\begin{Ex}
	Given adjunctions $L \dashv R : \D \to \C$ and $L' \dashv R' : \E \to \D$ there is a composite adjunction $L'L \dashv RR' : \E \to \C$, given by pasting the units and counits: 
\[\begin{tikzcd}
	\C & \C & \C && \C & \D & \E \\
	\D & \D & \C && \D & \D & \E \\
	\E & \D & {\C \rlap{ ,}} && \E & \E & {\E \rlap{ .}}
	\arrow["L"', from=1-1, to=2-1]
	\arrow["{L'}"', from=2-1, to=3-1]
	\arrow["L"', from=1-2, to=2-2]
	\arrow["R"', from=2-2, to=2-3]
	\arrow["1", from=2-3, to=3-3]
	\arrow["1", from=1-3, to=2-3]
	\arrow["1", from=1-1, to=1-2]
	\arrow["1", from=1-2, to=1-3]
	\arrow["{R'}"', from=3-1, to=3-2]
	\arrow["R"', from=3-2, to=3-3]
	\arrow["1", from=2-1, to=2-2]
	\arrow["1", from=2-2, to=3-2]
	\arrow["{\eta'}", shorten <=8pt, shorten >=8pt, Rightarrow, from=2-2, to=3-1]
	\arrow["\eta", shorten <=8pt, shorten >=8pt, Rightarrow, from=1-3, to=2-2]
	\arrow["L"', from=1-5, to=2-5]
	\arrow["{L'}"', from=2-5, to=3-5]
	\arrow["R"', from=1-6, to=1-5]
	\arrow["{R'}"', from=1-7, to=1-6]
	\arrow["1", from=1-7, to=2-7]
	\arrow["1", from=2-7, to=3-7]
	\arrow["1", from=3-7, to=3-6]
	\arrow["1", from=3-6, to=3-5]
	\arrow["{R'}"', from=2-7, to=2-6]
	\arrow["1", from=2-6, to=2-5]
	\arrow["1", from=1-6, to=2-6]
	\arrow["{L'}"', from=2-6, to=3-6]
	\arrow["\varepsilon"', shorten <=8pt, shorten >=8pt, Rightarrow, from=1-5, to=2-6]
	\arrow["{\varepsilon'}"', shorten <=8pt, shorten >=8pt, Rightarrow, from=2-6, to=3-7]
	\arrow[shorten <=8pt, shorten >=8pt, Rightarrow, no head, from=1-2, to=2-1]
	\arrow[shorten <=8pt, shorten >=8pt, Rightarrow, no head, from=2-3, to=3-2]
	\arrow[shorten <=8pt, shorten >=8pt, Rightarrow, no head, from=2-5, to=3-6]
	\arrow[shorten <=8pt, shorten >=8pt, Rightarrow, no head, from=1-6, to=2-7]
\end{tikzcd}\]
More explicitly, the unit of the composed adjunction is given by $(R\circ \eta' \circ L).\eta$, where $\circ$ denotes horizontal composition. Split reflections are closed under composition, and so the functor $\SplRef \to \Ar(\Cat)$ forms the arrow part of a double functor $\DSplRef \to \Sq(\Cat)$. 
\end{Ex}
\begin{Ex}
	Given split opfibrations $P : \A \to \B$ and $Q : \B \to \D$ the composition $Q.P$ is made into a split opfibration by first lifting along $Q$ and then along $P$. We can thus extend $\SplOpFib \to \Ar(\Cat)$ to a double functor $\DSplOpFib \to \Sq(\Cat)$. 
\end{Ex}

\subsection{Lifting}
A lifting operation for a pair of double functors $\DJ, \DK \to \Sq(\C)$ is one for their underlying functors over $\Ar(\C)$, satisfying a further coherence condition with respect to composition in $\DJ$ and $\DK$.  

\begin{Defn}
	Let $\DJ, \DK \to \Sq(\C)$; a $(\DJ, \DK)$-\emph{lifting operation} $\phi$ is a $(\DJ_1, \DK_1)$-lifting operation which respects composition in $\DJ$ and $\DK$, in the sense that given vertical composable morphisms $\f, \f'$ in $\DJ$, and $\g, \g'$ in $\DK$, as well as a lifting problem $(u, v) : f.f' \to g.g'$, we have $\phi_{\f, \g'}(\phi_{\f', \g.\g'}(u, v.f), \phi_{\f.\f', \g}(g'.u, v)) = \phi_{\f.\f', \g.\g'}(u, v)$:
\begin{equation}\label{eq:dc:verticalcondition}
	\begin{tikzcd}[column sep=scriptsize]
		\bullet && \bullet && \bullet && \bullet \\
		\bullet && \bullet & {=} \\
		\bullet && \bullet && \bullet && {\bullet \rlap{ .}}
		\arrow["{f'}"', from=1-1, to=2-1]
		\arrow["f"', from=2-1, to=3-1]
		\arrow["{g'}", from=1-3, to=2-3]
		\arrow["g", from=2-3, to=3-3]
		\arrow["{f.f'}"', from=1-5, to=3-5]
		\arrow["{g.g'}", from=1-7, to=3-7]
		\arrow["u", from=1-1, to=1-3]
		\arrow["v"', from=3-1, to=3-3]
		\arrow["u", from=1-5, to=1-7]
		\arrow["v"', from=3-5, to=3-7]
		\arrow["{\phi_{\f.\f', \g.\g'}}"{description}, dotted, from=3-5, to=1-7]
		\arrow[dotted, from=2-1, to=1-3]
		\arrow[dotted, from=3-1, to=2-3]
		\arrow["{\phi_{\f, \g'}}"{description}, dotted, from=3-1, to=1-3]
	\end{tikzcd}
\end{equation}
\end{Defn}

The assignment $(\DJ, \DK) \mapsto \CLift(\DJ, \DK)$ of double functors $\DJ, \DK$ to the class $\CLift(\DJ, \DK)$ of $(\DJ, \DK)$-lifting operations can again be made the object part of a functor 
\begin{equation}\label{eq:doublefuncliftfunctor}
    \CLift : (\Dbl/\Sq(\C))\op \times (\Dbl/\Sq(\C))\op \to \Set,
\end{equation}
which is again representable in both arguments, as shown in~\cite[Proposition~18]{bourke2016AlgebraicWeak}.
\begin{Prop}\label{prop:doublefunctor:galoisconnection}
    $\CLift$ is representable in both arguments, inducing an adjunction:
    \begin{equation}\label{eq:doublefunctor:galoisconnection}
	\begin{tikzcd}
		{(\Dbl/\Sq(\C))\op} & {\Dbl/\Sq(\C) \rlap{ .}}
		\arrow[""{name=0, anchor=center, inner sep=0}, "{\ldpf(-)}"', shift right=2, from=1-2, to=1-1]
		\arrow[""{name=1, anchor=center, inner sep=0}, "{(-)\ldpf}" {below, overlay}, shift right=2, from=1-1, to=1-2, overlay]
		\arrow["\dashv"{anchor=center, rotate=-90}, draw=none, from=0, to=1]
	\end{tikzcd}
    \end{equation}
	\vspace{-1.5ex}
\end{Prop}
\begin{proof}
The representing object of $\CLift(-, \DJ)$ is a double category $\ldpf \DJ \to \Sq(\C)$ of which the  vertical morphisms are pairs $(f, \phi_{f-})$ with $\phi_{f-} \in \CLift(f, \DJ)$ (where we consider $f$ to be a double functor $f: 1 \to \Sq(\C)$). Vertical composition in $\ldpf \DJ$ is performed as in (\ref{eq:dc:verticalcondition}), and a morphism $(f, \phi_{f-}) \to (g, \phi_{g-})$ is a square $(u, v) : f \to g$ which commutes with the lifting operations. The representing object $\DJ\ldpf \to \Sq(\C)$ of $\CLift(\DJ, -)$ is defined analogously. 
\end{proof}

We obtain the following analog of~(\ref{eq:funcJpfK}):
\begin{equation}\label{eq:doublefuncJpfK}
    \Dbl/\Sq(C)(\DJ, \ldpf \DK) \cong \CLift(\DJ, \DK) \cong \Dbl/\Sq(\C)(\DK, \DJ \ldpf).
\end{equation} 

\begin{Defn}\label{def:doublefuncliftingstructure}
    A triple $(\DJ, \phi, \DK)$ of double functors $\DJ, \DK \to \Sq(\C)$ and a $(\DJ, \DK)$-lifting operation $\phi$ is called a \emph{lifting structure}. A lifting structure $(\DJ, \phi, \DK)$ is called \emph{closed} when the functors $\phi_l : \DJ \to \ldpf \DK$ and $\phi_r : \DK \to \DJ \ldpf$ induced by (\ref{eq:doublefuncJpfK}) are invertible. 
\end{Defn} 
\begin{Ex}
    Let $\DJ \to \Sq(\C)$; applying~(\ref{eq:doublefuncJpfK}) to the identities on $\ldpf \DJ$ and $\DJ \ldpf$ yields lifting operations $\lcan \in \CLift(\ldpf \DJ, \DJ)$ and $\rcan \in \CLift(\DJ, \DJ\ldpf)$, and so we get, as we did in Example~\ref{ex:sscanliftingpairs}, two canonical lifting structures $(\ldpf \DJ, \lcan, \DJ)$ and $(\DJ, \rcan, \DJ\ldpf)$. Again, this yields lifting structure $(\ldpf \DJ, \rcan, (\ldpf \DJ)\ldpf)$ and $(\ldpf(\DJ \ldpf), \lcan, \DJ\ldpf)$ which are said to be fibrantly generated and cofibrantly generated by $\DJ$. 
\end{Ex}

The following is~\cite[Definition 3]{bourke2023OrthogonalApproach} of a lifting \awfs.
\begin{Defn}\label{def:awfs}
    An \emph{algebraic weak factorisation system} (\awfs) on a category $\C$ is a closed lifting structure $(\DJ, \phi, \DK)$ which satisfies the axiom of factorisation, requiring every morphism $f \in \C_1$ to admit a bi-universal factorisation $f = h.g$ for vertical morphisms $\g$ and $\h$ of $\bL$ and $\bR$ respectively. 
\end{Defn}

\begin{Ex}\label{ex:lawfs}
	The double functors $\DSplRef \to \Sq(\Cat)$ and $\DSplFib \to \Sq(\Cat)$ form part of an \awfs, when equipped with the lifting structure defined in Example \ref{ex:liftingstrforrari}. For the proof we refer to \cite[Examples 4 (ii)]{bourke2023OrthogonalApproach}.
\end{Ex}	

In what follows we will never need the axiom of factorisation, so we will focus on closed lifting structures instead.

\section{The Frobenius equivalence for double categories of maps}\label{sec:doublefunctors}

In this section we will finally establish a Frobenius equivalence for $\awfs$s. For this we follow the outline we sketched in Subsection \ref{sec:outlineFrobenius} and build on the work in Section \ref{sec:functions}. That is, we will establish double-categorical versions of base change and slicing, from which the Frobenius equivalence will be a direct corollary.

\subsection{Base Change} We start with base change. As noted in \cite[Proposition 21]{bourke2016AlgebraicWeak}, the adjunction in (\ref{eq:doublefunctor:galoisconnection}) can be extended to allow for change of base. This is done by considering a 2-category $\dladj$ which can be defined as the pullback:
\[\begin{tikzcd}[row sep=scriptsize]
	\dladj & {\Cat_\tladj} \\
	& \Cat \\
	{\Ar(\Dbl)} & {\Dbl \rlap{ .}}
	\arrow["\cod", from=3-1, to=3-2]
	\arrow[from=1-1, to=3-1]
	\arrow[from=1-1, to=1-2]
	\arrow["\lrcorner"{anchor=center, pos=0.125}, draw=none, from=1-1, to=3-2]
	\arrow["\Sq", from=2-2, to=3-2]
	\arrow[from=1-2, to=2-2]
\end{tikzcd}\] 
The 2-category $\dradj$ is defined dually. We now have an adjunction
\begin{equation}\label{eq:dcatorthogonalcob}
\begin{tikzcd}[ampersand replacement=\&]
	{(\dladj)^{\text{coop}}} \& {\dradj \rlap{ .}}
	\arrow[""{name=0, anchor=center, inner sep=0}, "{\ldpf(-)}"', shift right=2, from=1-2, to=1-1]
	\arrow[""{name=1, anchor=center, inner sep=0}, "{(-)\ldpf}"', shift right=2, from=1-1, to=1-2]
	\arrow["\dashv"{anchor=center, rotate=-90}, draw=none, from=0, to=1]
\end{tikzcd}
\end{equation}
extending that of (\ref{eq:doublefunctor:galoisconnection}), where $(-)^\text{coop}$ denotes reversal of both the 1- and 2-cells. 

\subsection{Slicing}
We now define a double categorical analog of the slicing operation that we saw in Section \ref{sec:function:slicing}. Given a double functor $\DJ \to \Sq(\C)$ and an object $A \in \C$ we define the slice double functor $\DJ/A \to \Sq(\C/A)$ by the following pullback:
\[\begin{tikzcd}[ampersand replacement=\&]
	{\DJ/A} \& \DJ \\
	{\Sq(\C/A)} \& {\Sq(\C) \rlap{ .}}
	\arrow[from=1-1, to=2-1]
	\arrow[from=1-2, to=2-2]
	\arrow["{\Sq(\dom)}", from=2-1, to=2-2]
	\arrow[from=1-1, to=1-2]
	\arrow["\lrcorner"{anchor=center, pos=0.125}, draw=none, from=1-1, to=2-2]
\end{tikzcd}\]
Pullbacks in $\Dbl$ are taken levelwise, so $(\DJ/A)_1 = \DJ_1/A$. Vertical composition in $\DJ/A$ is essentially just vertical composition in $\DJ$; a pair of vertical morphisms $(\g, b)$ and $(\f, a)$ in $\DJ/A$ is composable if $b = a.f$, and their composition is given by $(\f.\g, a)$. 

Like before, we have:
\begin{Prop}\label{prop:doublecat:slicels}
    Let $(\DJ, \phi, \DK)$ be a lifting structure on $\C$ and $A \in \C$ some object, then there is a lifting structure $(\DJ/A, \phi/A, \DK/A)$ on $\C/A$. 
\end{Prop}
\begin{proof}
	Let $\phi/A$ be the lifting operation defined in Proposition \ref{prop:function:slicels}: it respects vertical composition whenever $\phi$ does.
\end{proof}

We would like to extend this result to:
\begin{Prop}\label{prop:doublecat:slicecls}
    If $(\DJ, \phi, \DK)$ is a closed lifting structure then so is $(\DJ/A, \phi/A, \DK/A)$.
\end{Prop}

For this we need an analog of the results of Section \ref{sec:function:slicing}: that for $\DJ \to \Sq(\C)$ there are isomorphisms $(\DJ/A)^\dpf \cong (\DJ^\dpf)/A$ and $({^\dpf}\DJ)/A \cong {^\dpf}(\DJ/A)$. As before, only the former of these exists in general, for the latter we need an additional assumption.

\begin{Prop}\label{prop:doublecat:scandcommrl}
    For $\DJ \to \Sq(\C)$, and $A \in \C$, there are double functors \[ (\lambda/A)_l : (\ldpf \DJ)/A \to \ldpf(\DJ/A) \quad \mbox{ and } \quad (\rho / A)_r : (\DJ\ldpf)/A \to (\DJ/A)\ldpf, \] with the latter double functor $(\rho / A)_r$ being an isomorphism. 
\end{Prop} 
\begin{proof}
	At this point this should be routine.
\end{proof}

\begin{Defn}
    A double functor $\DJ \to \Sq(\C)$ is a \emph{comprehension double category} if $\DJ_1 \to \Ar(\C)$ is a comprehension category as in Definition \ref{defn:function:comprcat}, and, moreover, the squares in $\DJ$ that are cartesian with respect to the Grothendieck fibration $\cod.\DJ_1$ are closed under vertical composition. That is, if we have two squares in $\DJ$ that can be composed vertically as in
	\begin{displaymath} 
		\begin{tikzcd}
			\bullet \ar[d, "{\g'}"'] \ar[r, "\u"] & \bullet \ar[d, "\g"] \\
			\bullet \ar[d, "{\h'}"'] \ar[r, "\v"] & \bullet \ar[d, "\h"] \\
			\bullet \ar[r, "\w"] & \bullet,
		\end{tikzcd}
	\end{displaymath}
	in which the top square is cartesian over $v$ and the bottom square cartesian over $w$, then the vertically composed square is cartesian over $w$.
\end{Defn}
\begin{Prop}\label{prop:doublecat:comprehensiondouble}
    $\DJ\ldpf \to \Sq(\C)$ is a comprehension double category for any $\DJ \to \Sq(\C)$.
\end{Prop}
\begin{proof}
    This follows from the fact that a square in $\DJ\ldpf$ is cartesian precisely when its image under $(\DJ\ldpf)_1 \to \Ar(\C)$ is a pullback. 
\end{proof}

\begin{Prop}\label{prop:doublecat:commll}
    Let $\DJ \to \Sq(\C)$ be a comprehension double category, and $A \in \C$ be some object; then the double functor $(\lambda / A)_l$ of Proposition~\ref{prop:doublecat:scandcommrl} is an isomorphism.
\end{Prop}
\begin{proof}
	We show that the lifting operation defined in (\ref{eq:liftop}) satisfies the vertical condition. To this end we consider a vertical composition $\h.\g$ in $\DJ$ and a square $(u, v) : f \to h.g$.:
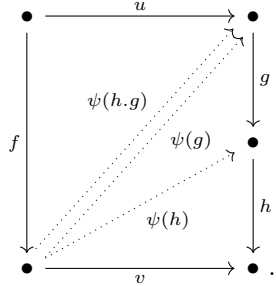
\[\begin{tikzcd}[column sep=large,row sep=huge]
	\bullet && \bullet \\
	&& \bullet \\
	\bullet && {\bullet \rlap{ .}}
	\arrow["g", from=1-3, to=2-3]
	\arrow["h", from=2-3, to=3-3]
	\arrow["f"', from=1-1, to=3-1]
	\arrow["u", from=1-1, to=1-3]
	\arrow["v"', from=3-1, to=3-3]
	\arrow["{\psi(h)}"', dotted, from=3-1, to=2-3]
	\arrow["{\psi(g)}"'{pos=0.6}, shift right=1, dotted, from=3-1, to=1-3]
	\arrow["{\psi(h.g)}"{pos=0.6}, shift left=1, dotted, from=3-1, to=1-3]
\end{tikzcd}\]
We use the abbreviations $\psi(h) = \psi_{f, \h}(g.u, v)$, $\psi(g) = \psi_{f, \g}(u, \psi(h))$, and $\psi(h.g) = \psi_{f, \h.\g}(u, v)$, so we should show $\psi(h.g) = \psi(g)$. The diagrams below respectively show how these are computed:
\[\begin{tikzcd}
	\bullet & \bullet & \bullet & \bullet & \bullet & \bullet & \bullet & \bullet & \bullet \\
	\bullet & \bullet & {\bullet \rlap{ ,}} & \bullet & \bullet & {\bullet \rlap{ ,}} & \bullet & \bullet & {\bullet \rlap{ .}}
	\arrow["f"', from=1-1, to=2-1]
	\arrow["{v^*h}", from=1-2, to=2-2]
	\arrow["{v^+}", from=1-2, to=1-3]
	\arrow["h", from=1-3, to=2-3]
	\arrow["v"', from=2-2, to=2-3]
	\arrow[no head, from=2-1, to=2-2]
	\arrow["\alpha", from=1-1, to=1-2]
	\arrow["{\phi(h)}"{description}, dotted, from=2-1, to=1-2]
	\arrow["f"', from=1-4, to=2-4]
	\arrow["{\psi(h)^*g}", from=1-5, to=2-5]
	\arrow["g", from=1-6, to=2-6]
	\arrow["f"', from=1-7, to=2-7]
	\arrow["{v^*(h.g)}", from=1-8, to=2-8]
	\arrow["{h.g}", from=1-9, to=2-9]
	\arrow["\beta", from=1-4, to=1-5]
	\arrow["{\psi(h)^+}", from=1-5, to=1-6]
	\arrow[no head, from=2-4, to=2-5]
	\arrow["\gamma", from=1-7, to=1-8]
	\arrow["{v^{++}}", from=1-8, to=1-9]
	\arrow[no head, from=2-7, to=2-8]
	\arrow["v"', from=2-8, to=2-9]
	\arrow["{\psi(h)}"', from=2-5, to=2-6]
	\arrow["{\phi(g)}"{description}, dotted, from=2-4, to=1-5]
	\arrow["{\phi(h.g)}"{description}, dotted, from=2-7, to=1-8]
\end{tikzcd}\]
So for instance $\psi(h) = v^+.\phi(h)$, where $\phi(h)$ is an abbreviation for $\phi_{(f,a),(v^*\h,a)}(\alpha, 1)$, which is the lift obtained from the $\DJ/A$ lifting operation of $(f, a)$ applied to the $\DJ/A$ map $(v^*\h, a)$. 

Our assumption that cartesian square are closed under vertical composition implies that the cartesian square $v^*(\h.\g) \to \h.\g$ can be assumed to be the vertical composition of two squares
\begin{displaymath} 
	\begin{tikzcd}
		\bullet \ar[d, "{v^*\g}"] \ar[r, "\v^{++}"] \ar[dd, bend right, shift right, "{v^*(\h.\g)}"'] & \bullet \ar[d, "\g"] \ar[dd, , shift left, bend left, "{\h.\g}"]\\
		\bullet \ar[d, "{v^*\h}"] \ar[r, "\v^{+}"'] & \bullet \ar[d, "\h"] \\
		\bullet \ar[r, "\v"'] & \bullet,
	\end{tikzcd}
\end{displaymath}
with $\v^{++}\v^+ : v^*\g \to \g$ a cartesian morphism in $\DJ$ over $v^+$. Since $v^*\g$ is a vertical map in $\DJ$, we get a lift $\phi'(g) = \phi_{(f,a), (v^*\g,a.v^*h)}(\gamma, \phi(h))$, and by the vertical condition $\phi(h.g) = \phi'(g)$. 

Furthermore, because $\v^{++}\v^+$ is cartesian, there is a square $\i\j : \psi(h)^*\g \to v^*\g$ in $\DJ$:
\[\begin{tikzcd}
	\bullet & \bullet & \bullet \\
	\bullet & \bullet & \bullet 
	\arrow["{\psi(h)^*{\g}}"', from=1-1, to=2-1]
	\arrow["{v^*{\g}}", from=1-2, to=2-2]
	\arrow["{\g}", from=1-3, to=2-3]
	\arrow["{\i}", from=1-1, to=1-2]
	\arrow["{\j}"', from=2-1, to=2-2]
	\arrow["{\v^{++}}", from=1-2, to=1-3]
	\arrow["{\v^+}"', from=2-2, to=2-3]
\end{tikzcd}\]
with $j = \phi(h)$. Since $\gamma = i.\beta$ we get that $\phi(h.g) = \phi'(g) = i.\phi(g)$, and so \[\psi(h.g) = v^{++}.\phi(h.g) = v^{++}.i.\phi(g) = \psi(h)^+.\phi(g) = \psi(g). \qedhere\] 
\end{proof}

\subsection{Frobenius equivalence} The following is now a direct analogue of the results from Subsection \ref{sec:function:frobenius}.

\begin{Defn}\label{def:dfunctor:frobstruc}
    Let $f: A \to B$ be a map in $\C$, and $(\DJ, \DK)$ a pair of double functors over $\Sq(\C)$. A \emph{Frobenius structure} $\f^*$ on $f$ for $(\DJ, \DK)$ is a filler as on the left of: 
\[\begin{tikzcd}[ampersand replacement=\&]
	{\DJ/B} \& {\DJ/A} \& {\DK/A} \& {\DK/B} \\
	{\Sq(\C/B)} \& {\Sq(\C/A) \rlap{ ,}} \& {\Sq(\C/A)} \& {\Sq(\C/B) \rlap{ .}}
	\arrow["{\Sq(f^*)}", from=2-1, to=2-2]
	\arrow[from=1-1, to=2-1]
	\arrow[from=1-2, to=2-2]
	\arrow["{\f^*}", dotted, from=1-1, to=1-2]
	\arrow["{\Sq(f_*)}", from=2-3, to=2-4]
	\arrow[from=1-3, to=2-3]
	\arrow[from=1-4, to=2-4]
	\arrow["{\f_*}", dotted, from=1-3, to=1-4]
\end{tikzcd}\]
Similarly, a \emph{pushforward structure} $\f_*$ on $f$ is a filler as on the right above. 
\end{Defn}

\begin{Thm}\label{thm:dffequiv}
	For a map $f: A \to B$ in \ct{C}, and a closed lifting structure $(\DJ, \phi, \DK)$, there is a bijection between Frobenius and pushforward structures on $f$ for $(\DJ, \DK)$.
\end{Thm}
\begin{proof}
	Combining the previous results we get
	\begin{align*}
	\dladj(\DJ/B, \DJ/A) &\cong \dladj(\DJ/B, (\DK/A)^\dpf) \\
	&\cong \dradj(\DK/A, {^\dpf}(\DJ/B)) \\
	&\cong \dradj(\DK/A, \DK/B), 
	\end{align*}
	where we consider morphisms in $\dladj$ and $\dradj$ w.r.t.\ $f^* \dashv f_*$.
\end{proof}

\begin{Defn}\label{def:dfrobstruc}
	A \emph{Frobenius structure} for a pair $(\DJ, \DK)$ is an assignment of a Frobenius structure $\f^*$ to every vertical morphism $\f$ of $\DK$.
\end{Defn}

\begin{Rk}
	Let us interpret the meaning of a Frobenius structure on $f : A \to B$ when the double categories are given by the pair $(\DLcoalg, \DRalg)$ underlying an \awfs $(\L, \R)$. The double category $\DCoalg{\L}/A$ is then formed as on the left below:
\[\begin{tikzcd}[ampersand replacement=\&]
	{\DCoalg{\L}/A} \& {\DCoalg{\L}} \& {\DCoalg{\L}/B} \& {\DCoalg{\L}/A} \\
	{\Sq(\C/A)} \& {\Sq(\C) \rlap{ ,}} \& {\Sq(\C/B)} \& {\Sq(\C/A) \rlap{ .}}
	\arrow[from=1-1, to=2-1]
	\arrow["{\Sq(\dom)}", from=2-1, to=2-2]
	\arrow[from=1-1, to=1-2]
	\arrow[from=1-2, to=2-2]
	\arrow["\lrcorner"{anchor=center, pos=0.125}, draw=none, from=1-1, to=2-2]
	\arrow["{\Sq(f^*)}", from=2-3, to=2-4]
	\arrow[from=1-4, to=2-4]
	\arrow[from=1-3, to=2-3]
	\arrow["{\f^*}", from=1-3, to=1-4]
\end{tikzcd}\]
	This situation is discussed more generally in \cite[Section 4.5]{bourke2016AlgebraicWeak}, and it is shown there that the pullback along $\Sq(\dom)$ yields a \emph{slice} \awfs $(\L/A, \R/A)$ on $\C/A$ such that $(\mathsf{L}/A)\text-\mathbb C\cat{oalg} \cong \DLcoalg/A$. This means that the square of the Frobenius structure, pictured on the right above, is in the image of the semantics functor $\DCoalg{(-)}$. Therefore, as this functor is full \cite[Proposition 2]{bourke2016AlgebraicWeak}, a Frobenius structure on $f$ can be understood as making $f^*$ an oplax morphism of slice \awfs $(\L/B, \R/B) \to (\L/A, \R/A)$. Dually, a pushforward structure $\f_*$ on $f$ corresponds to a lax morphism of slice \awfs.
\end{Rk}

\subsection{An example of split reflections and split opfibrations} As an example of a pair with a Frobenius structure, we have the following adaptation of Proposition~\ref{prop:functor:frobstruc}.

\begin{Prop}\label{prop:dfrobstruc}
	The pair $(\DSplRef, \DSplOpFib)$ admits a Frobenius structure. 
\end{Prop}
\begin{proof}
	We extend the Frobenius structure for a split opfibration $P : \A \to \B$ of Proposition~\ref{prop:functor:frobstruc} to one for this pair by showing $\P^*$ preserves composition of split reflections, yielding:
\[\begin{tikzcd}
	{\DSplRef/\B} & {\DSplRef/\A} \\
	{\Sq(\Cat/\B)} & {\Sq(\Cat/\A) \rlap{ .}}
	\arrow[from=1-1, to=2-1]
	\arrow["{\mathbf{P}^*}", from=1-1, to=1-2]
	\arrow[from=1-2, to=2-2]
	\arrow["{\Sq(P^*)}", from=2-1, to=2-2]
\end{tikzcd}\]
To verify this, we consider the situation sketched below: 
\[\begin{tikzcd}
	{\A\times_\B\E} & \E & {\A\times_\B\E} \\
	{\A\times_\B\D} & \D \\
	{\A\times_\B\C} & \C & {\A\times_\B\C} \\
	\A & \B & {\A \rlap{ .}}
	\arrow["P"', from=4-1, to=4-2]
	\arrow["P", from=4-3, to=4-2]
	\arrow["{\pi_\C}"', from=3-3, to=3-2]
	\arrow["{\pi_\D}", from=2-1, to=2-2]
	\arrow["{\pi_\E}", from=1-1, to=1-2]
	\arrow["{\pi_\E}"', from=1-3, to=1-2]
	\arrow["U", from=3-2, to=4-2]
	\arrow["{P^*U}", shift left, from=3-1, to=4-1]
	\arrow["F"', shift right, from=3-1, to=4-1]
	\arrow["F'"', curve={height=18pt}, shift right, from=2-1, to=4-1]
	\arrow["{P^*U}", shift left, from=3-3, to=4-3]
	\arrow["{F''}"', shift right, from=3-3, to=4-3]
	\arrow["{P^*(R.R')}", shift left, from=1-3, to=3-3]
	\arrow["{G''}", shift left, from=3-3, to=1-3]
	\arrow["{P^*R}", shift left, from=2-1, to=3-1]
	\arrow["G", shift left, from=3-1, to=2-1]
	\arrow["{P^*R'}", shift left, from=1-1, to=2-1]
	\arrow["{G'}", shift left, from=2-1, to=1-1]
	\arrow["R", shift left, from=2-2, to=3-2]
	\arrow["{\pi_\C}", from=3-1, to=3-2]
	\arrow["L", shift left, from=3-2, to=2-2]
	\arrow["{R'}", shift left, from=1-2, to=2-2]
	\arrow["{L'}", shift left, from=2-2, to=1-2]
\end{tikzcd}\]
On the left we have the composition of the pullbacks $G'.G \dashv P^*R.P^*R'$, and on the right the pullback of the composition $G'' \dashv P^*(R.R')$. We should show that $G'' = G'.G$, and that the unit $\theta''$ of $G'' \dashv P^*(R'R)$ is equal to $(P^*R \circ \theta' \circ G).\theta$. 

These equalities hold precisely because $P$ is split, since this tells us that for a pair $(a, c) \in \A \times_\B \C$ we have $\smash{\underline{UR\eta'_{Lc}.U\eta_C}} = \smash{\underline{UR\eta'_{Lc}}}.\underline{U\eta_c}$:
\begin{equation}\label{eq:splitopfib}
	\begin{tikzcd}
		a & {(U\eta_c)_!c} & {(UR\eta'_{Lc})_!(U\eta_c)_!a} \\
		Uc & URLc & {URR'L'Lc \rlap{ .}}
		\arrow["{U\eta_c}"', from=2-1, to=2-2]
		\arrow[dashed, no head, from=1-1, to=2-1]
		\arrow["\underline{U\eta_c}"', from=1-1, to=1-2]
		\arrow[dashed, no head, from=1-2, to=2-2]
		\arrow["{UR\eta'_{Lc}}"', from=2-2, to=2-3]
		\arrow[dashed, no head, from=1-3, to=2-3]
		\arrow["\underline{UR\eta'_{Lc}}"', from=1-2, to=1-3]
		\arrow["\underline{UR\eta'_{Lc}.U\eta_C}", curve={height=-18pt}, from=1-1, to=1-3]
	\end{tikzcd}
\end{equation}
To begin we note that $F'(a, d) = (UR\eta'_d)_!a$, and so:
\begin{align*}
	F''(a, c) &= (U((R\circ\eta'\circ L).\eta)_c)_!a \\
	&= (UR\eta'_{Lc}.U\eta_c)_!a \\
	&= (UR\eta'_{Lc})_!(U\eta_c)_!a \\
	&= F'((U\eta_c)_!a, Lc) \\
	&= F'G(a, c);
\end{align*} 
which is to say $F'.G = F''$. Therefore, \[G'.G = (F' \times_\B L'.\pi_\D).G = F'.G \times_\B L'.\pi_\D.G = F'' \times_\B L'.L.\pi_\C = G''.\] The other equality $\theta'' = (P^*R\circ\theta'\circ G).\theta$ also follows readily from~(\ref{eq:splitopfib}):
\begin{align*}
	((P^*R\circ\theta'\circ G).\theta)_{(a, c)} &= P^*R(\theta'_{G(a, c)}).\theta_{(a,c)} \\
	&= (\smash{\underline{UR\eta'_{Lc}}}, R\eta'_{Lc}).(\smash{\underline{U\eta_c}},\eta_c) \\
	&= (\smash{\underline{UR\eta'_{Lc}}}.\smash{\underline{U\eta_c}}, R\eta'_{Lc}.\eta_c) \\
	&= (\smash{\underline{U((R \circ \eta'\circ L).\eta)_c}}, ((R \circ \eta' \circ L).\eta)_c) \\
	&= \theta''_{(a,c)}. \qedhere
\end{align*}
\end{proof}

\section{The Beck--Chevalley condition}

Ultimately we are interested in constructing models of type theory using the right adjoint method for splitting comprehension categories, as used in~\cite{gambino2023ModelsMartinLof}. In order for this method to work the right maps should satisfy a number of stability conditions, as outlined in~\cite[Chapter 2]{larrea2018ModelsDependent}. The stability condition for the $\Pi$-types is ensured by a Beck--Chevalley condition (cf.\ Proposition~\ref{prop:ptypes}) which we will now phrase. 

\subsection{The categorical case}

Consider a closed lifting structure $(\J, \phi, \K)$ with $\J$ and $\K$ being categories of maps. For the purpose of interpreting dependent product types we want this pair to have a pushforward structure. However, in order for this interpretation to satisfy the coherence axioms of the theory with respect to substitution, this pushforward structure should respect morphisms of $\K$ whose underlying square in the ambient category $\C$ is a pullback square. 

More specifically, consider $\f, \g, \h \in \K$ such that $g$ and $h$ are composable, and a morphism $\mathbf{uv} : \f \to \g$ whose underlying square in $\C$ is a pullback square:
\begin{equation}\label{eq:BCobj}
\begin{tikzcd}
	& \bullet \\
	A & C &[1cm]& \bullet &[-.5cm]&[-.5cm] \bullet \\
	B & {D \rlap{ ,}} &&& {B \rlap{ .}}
	\arrow["f"', from=2-1, to=3-1]
	\arrow["g", from=2-2, to=3-2]
	\arrow["h", from=1-2, to=2-2]
	\arrow["u", from=2-1, to=2-2]
	\arrow["v"', from=3-1, to=3-2]
	\arrow["\lrcorner"{anchor=center, pos=0.125}, draw=none, from=2-1, to=3-2]
	\arrow["{f_*u^*h}"', from=2-4, to=3-5]
	\arrow["{v^*g_*h}", from=2-6, to=3-5]
	\arrow["{\beta^{-1}_h}", from=2-4, to=2-6]
\end{tikzcd}
\end{equation}
This gives composite parallel adjunctions $u_!f^* \dashv f_*u^*, g^*v_! \dashv v^*g_* : \C/C \to \C/B$:
\begin{equation}\label{eq:bcmateadjs}
	\begin{tikzcd}[column sep=.5em]
		{\C/B} & \bot & {\C/A} & \bot & {\C/C \rlap{ ,}} &[.5cm]& {\C/B} & \bot & {\C/D} & \bot & {\C/C \rlap{ .}}
		\arrow["{f^*}", shift left=3, from=1-1, to=1-3]
		\arrow["{f_*}", shift left=3, from=1-3, to=1-1]
		\arrow["{u_!}", shift left=3, from=1-3, to=1-5]
		\arrow["{u^*}", shift left=3, from=1-5, to=1-3]
		\arrow["{v_!}", shift left=3, from=1-7, to=1-9]
		\arrow["{v^*}", shift left=3, from=1-9, to=1-7]
		\arrow["{g^*}", shift left=3, from=1-9, to=1-11]
		\arrow["{g_*}", shift left=3, from=1-11, to=1-9]
	\end{tikzcd}
\end{equation}
These induce a mateship correspondence between natural transformations with signatures $u_!f^* \To g^*v_!$ and $v^*g_* \To f_*u^*$. There is a canonical map $\alpha : u_!f^* \To g^*v_!$ (itself obtained as a mate of the identity natural transformation $g_!u_! = v_!f_!$) of which the component $\alpha_w$ at an object $w$ in $\C/B$ is given as follows: 
\[\begin{tikzcd}
	\bullet \\
	& \bullet & A & C \\
	& \bullet & B & {D \rlap{ .}}
	\arrow["f"', from=2-3, to=3-3]
	\arrow["g", from=2-4, to=3-4]
	\arrow["u", from=2-3, to=2-4]
	\arrow["v"', from=3-3, to=3-4]
	\arrow["\lrcorner"{anchor=center, pos=0.125}, draw=none, from=2-3, to=3-4]
	\arrow["w"', from=3-2, to=3-3]
	\arrow["{f^*w}", from=2-2, to=2-3]
	\arrow[from=2-2, to=3-2]
	\arrow["\lrcorner"{anchor=center, pos=0.125}, draw=none, from=2-2, to=3-3]
	\arrow["{g^*v_!w}", curve={height=-12pt}, from=1-1, to=2-4]
	\arrow[curve={height=12pt}, from=1-1, to=3-2]
	\arrow["{\alpha_w}"', dotted, from=2-2, to=1-1]
\end{tikzcd}\]
We denote the mate of $\alpha$ with signature $v^*g_* \To f_*u^*$ by $\beta$. These transformations $\alpha$ and $\beta$ exist regardless of whether $(u, v)$ is a pullback, but when it is they are both isomorphisms; see e.g.~\cite[Lemma~3.5]{hazratpour20242categoricalProof}. We call $\beta$ the Beck--Chevalley isomorphism.

In the setup of the left diagram of~(\ref{eq:BCobj}), we thus get a triangle pictured as the right diagram of~(\ref{eq:BCobj}). When $(\J, \phi, \K)$ is a closed lifting structure, the right class is always closed under pullbacks, and so if it comes with a Frobenius structure then the morphisms $f_*u^*h$ and $v^*g_*h$ have $\K$ structure. What we need for the interpretation of dependent products is that $(\beta_h^{-1}, 1)$ underlies a morphism of $\K$ maps. The goal of this section is to precisely state this condition, and to show it is equivalent to a condition that is easier to check in practice. The way we do this is similar to the statement and proof of~\cite[Proposition 6.7]{gambino2017FrobeniusCondition}. 

The first step will be to generalize to the case where $h$ is not an object but an arrow of $\C/C$, i.e.\ we consider $(\h ,w) \in \K/C$:
\begin{equation}\label{eq:BCarr}
	\begin{tikzcd}
		& \bullet \\
		& \bullet &[1cm]& \bullet &[-.5cm]&[-.5cm] \bullet \\
		A & C && \bullet && \bullet \\
		B & {D \rlap{ ,}} &&& {B \rlap{ .}}
		\arrow["f"', from=3-1, to=4-1]
		\arrow["g", from=3-2, to=4-2]
		\arrow["w", from=2-2, to=3-2]
		\arrow["u", from=3-1, to=3-2]
		\arrow["v"', from=4-1, to=4-2]
		\arrow["\lrcorner"{anchor=center, pos=0.125}, draw=none, from=3-1, to=4-2]
		\arrow["{f_*u^*w}"', from=3-4, to=4-5]
		\arrow["{v^*g_*w}", from=3-6, to=4-5]
		\arrow["{\beta^{-1}_{w}}", from=3-4, to=3-6]
		\arrow["h", from=1-2, to=2-2]
		\arrow["{f_*u^*h}"', from=2-4, to=3-4]
		\arrow["{v^*g_*h}", from=2-6, to=3-6]
		\arrow["{\beta^{-1}_{w.h}}", from=2-4, to=2-6]
	\end{tikzcd}
\end{equation}
The square $(\beta^{-1}_{w.h}, \beta^{-1}_w)$ on the right above is the component $\Ar(\beta^{-1})_{(h, w)}$ of the natural transformation $\Ar(\beta^{-1}): \Ar(f_*u^*) \To \Ar(v^*g_*)$. For a closed lifting structure $(\J, \phi, \K)$ an isomorphism in $\Ar(\C)$ underlies a morphism of $\J$ or $\K$ maps if and only if its inverse does, so we will focus on $\beta$ instead of $\beta^{-1}$. To state what it means for $\Ar(\beta)_{(h, w)}$ to underlie a morphism of $\K$ maps we use the following definition. 
\begin{Defn}\label{def:catnatlift}
	Let $\F, \G : \J \to \K$ be lifts of $F, G : \C \to \D$ for some $\J\to \Ar(\C)$ and $\K \to \Ar(\D)$, and $\mu : F \To G$ a natural transformation. We say that $\boldsymbol{\mu} : \F \To \G$ is a \emph{lift} of $\mu$ if $\K \circ \boldsymbol{\mu} = \Ar(\mu) \circ \J$, meaning its components are over those of $\mu$:
	\[\begin{tikzcd}[column sep=4em,row sep=3em]
		\J & \K \\
		{\Ar(\C)} & {\Ar(\D)} \rlap{ .}
		\arrow[""{name=0, anchor=center, inner sep=0}, "{\Ar(F)}", shift left=3, from=2-1, to=2-2]
		\arrow[""{name=1, anchor=center, inner sep=0}, "{\Ar(G)}"', shift right=3, from=2-1, to=2-2]
		\arrow[from=1-1, to=2-1]
		\arrow[from=1-2, to=2-2]
		\arrow[""{name=2, anchor=center, inner sep=0}, "{\F}", shift left=3, from=1-1, to=1-2]
		\arrow[""{name=3, anchor=center, inner sep=0}, "{\G}"', shift right=3, from=1-1, to=1-2]
		\arrow["{\,\Ar(\mu)}", shift right=4, shorten <=2pt, shorten >=2pt, Rightarrow, from=0, to=1]
		\arrow["{\,\boldsymbol{\mu}}", shift right, shorten <=2pt, shorten >=2pt, Rightarrow, from=2, to=3]
	\end{tikzcd}\]
\end{Defn}

Note that a composition functor like $u_! : \C/A \to C/C$ always lifts to a functor $\bu_! : \J/A \to \J/C: (\f, a) \mapsto (\f, u.a)$. Since $(\J,\K)$ has a Frobenius structure, we thus obtain lifts $\u_!\f^*$ and $\g^*\v_!$ as on the left of: 
\begin{equation}\label{eq:catalpha}
	\begin{tikzcd}[sep=3em]
		{\J/B} & {\J/C} & {\K/C} & {\K/B} \\
		{\Ar(\C/B)} & {\Ar(\C/C)} \rlap{ ,} & {\Ar(\C/C)} & {\Ar(\C/B)} \rlap{ .}
		\arrow[from=1-1, to=2-1]
		\arrow[from=1-2, to=2-2]
		\arrow[""{name=0, anchor=center, inner sep=0}, "{\u_!\f^*}", shift left=3, from=1-1, to=1-2]
		\arrow[""{name=1, anchor=center, inner sep=0}, "{\Ar(u_!f^*)}", shift left=3, from=2-1, to=2-2]
		\arrow[from=1-4, to=2-4]
		\arrow[""{name=2, anchor=center, inner sep=0}, "{\v^*\g_*}", shift left=3, from=1-3, to=1-4]
		\arrow[""{name=3, anchor=center, inner sep=0}, "{\Ar(v^*g_*)}", shift left=3, from=2-3, to=2-4]
		\arrow[from=1-3, to=2-3]
		\arrow[""{name=4, anchor=center, inner sep=0}, "{\Ar(g^*v_!)}"', shift right=3, from=2-1, to=2-2]
		\arrow[""{name=5, anchor=center, inner sep=0}, "{\Ar(f_*u^*)}"', shift right=3, from=2-3, to=2-4]
		\arrow[""{name=6, anchor=center, inner sep=0}, "{\f_*\u^*}"', shift right=3, from=1-3, to=1-4]
		\arrow[""{name=7, anchor=center, inner sep=0}, "{\g^*\v_!}"', shift right=3, from=1-1, to=1-2]
		\arrow["{\,\Ar(\alpha)}", shift right=4, shorten <=2pt, shorten >=2pt, Rightarrow, from=1, to=4]
		\arrow["{\,\bb}", shift right=2, shorten <=2pt, shorten >=2pt, Rightarrow, from=2, to=6]
		\arrow["{\,\Ar(\beta)}", shift right=4, shorten <=2pt, shorten >=2pt, Rightarrow, from=3, to=5]
		\arrow["{\,\ba}", shift right=2, shorten <=2pt, shorten >=2pt, Rightarrow, from=0, to=7]
	\end{tikzcd}
\end{equation}

Applying the reasoning of Proposition~\ref{prop:sc:bcgaloisconnection} with respect to the adjunctions~(\ref{eq:bcmateadjs}), we obtain corresponding lifts $\v^*\g_*$ and $\f_*\u^*$ as on the right above. The condition that the square $(\beta_{w.h}, \beta_w)$ is a morphism of $\K$ maps is now expressed by requiring the existence of a lift $\bb$ as depicted on the right above, due to the requirement that $(\K/B)\circ\bb = \Ar(\beta)\circ(\K/C)$. Note that---since the lifting structure $(\J, \phi, \K)$ is closed---the existence of this lift is not additional structure, but is instead just a property of $\beta$. In sum, we have the following definition. 
\begin{Defn}\label{def:functor:BC}
	A Frobenius structure for a closed structure $(\J, \phi, \K)$ satisfies the \emph{Beck--Chevalley condition} when for every $\mathbf{uv} : \f \to \g$ in $\K$ overlying a pullback square $(u, v)$, the Beck--Chevalley isomorphism $\beta$ lifts to a transformation $\bb: \v^*\g_* \To \f_*\u^*$. 
\end{Defn}
\begin{Prop}\label{prop:catbc}
	For a closed lifting structure $(\J, \phi, \K)$ with a Frobenius structure, the Beck--Chevalley isomorphism $\beta$ lifts in the sense of~(\ref{eq:catalpha}) iff its mate $\alpha$ does. 
\end{Prop}
\begin{proof}
	By~\cite[Proposition 5.8]{gambino2017FrobeniusCondition}, using that $\J/C \cong \lpf(\K/C)$ and $\K/B \cong (\J/B)\lpf$. 
\end{proof}

For future reference we expand on what it means for $\alpha$ to lift in the sense of Proposition~\ref{prop:catbc}. Considering the pullback square $(u, v)$ in~(\ref{eq:BCarr}) and $(w, b) \in \Ar(\C/B)$, the component $\Ar(\alpha)_{(w, b)}$ is given by the square $(\alpha_{b.w}, \alpha_b) : u_!f^*(w, b) \to g^*v_!(w, b)$:
\begin{equation}\label{eq:ARalphawb}
	\begin{tikzcd}[sep=2.25em]
		\bullet & \bullet & \bullet \\
		\bullet & \bullet & \bullet \\
		B & A & C \\[-.5cm]
		& D \rlap{ .}
		\arrow["{f^*b}"', from=2-2, to=3-2]
		\arrow["{g^*(v.b)}", from=2-3, to=3-3]
		\arrow["u", from=3-2, to=3-3]
		\arrow["{f^*w}"', from=1-2, to=2-2]
		\arrow["{\alpha_b}", dotted, from=2-2, to=2-3]
		\arrow["{g^*w}", from=1-3, to=2-3]
		\arrow["{\alpha_{b.w}}", dotted, from=1-2, to=1-3]
		\arrow["f"', from=3-2, to=3-1]
		\arrow["b"', from=2-1, to=3-1]
		\arrow["w"', from=1-1, to=2-1]
		\arrow[from=2-2, to=2-1]
		\arrow[from=1-2, to=1-1]
		\arrow["v"', from=3-1, to=4-2]
		\arrow["g", from=3-3, to=4-2]
	\end{tikzcd}
\end{equation}
If in addition $(\w, b) \in \J/B$, then $f^*w$ and $g^*w$ underlie $\J$ maps $\f^*\w$ and $\g^*\w$, and $\ba$ should produce a $\J$ morphism $\ba_{(\w, b)} : \f^*\w \to \g^*\w$ overlying $(\alpha_{b.w}, \alpha_b)$. 

\subsection{The double categorical case} 

Next, we phrase an analogous version of Definition~\ref{def:functor:BC} for a pair of double categories. More specifically, we consider the case where (\ref{eq:BCarr}) underlies a square $\uv : \f \to \g$ in $\DK$ for some closed lifting structure $(\DJ, \phi, \DK)$ with a Frobenius structure. To start, we define lifts of double natural transformations. 
\begin{Defn}\label{def:dcatnatlift}
	Let $\F, \G : \DJ \to \DK$ be lifts of $F, G : \C \to \D$ for some $\DJ\to \Ar(\C)$ and $\DK \to \Ar(\D)$, and $\mu : F \To G$ a natural transformation. We say that $\boldsymbol{\mu} : \F \To \G$ is a \emph{lift} of $\mu$ if $\DK \circ \boldsymbol{\mu} = \Sq(\mu) \circ \DJ$:
	\[\begin{tikzcd}[column sep=4em,row sep=3em]
		\DJ & \DK \\
		{\Sq(\C)} & {\Sq(\D)} \rlap{ .}
		\arrow[""{name=0, anchor=center, inner sep=0}, "{\Sq(F)}", shift left=3, from=2-1, to=2-2]
		\arrow[""{name=1, anchor=center, inner sep=0}, "{\Sq(G)}"', shift right=3, from=2-1, to=2-2]
		\arrow[from=1-1, to=2-1]
		\arrow[from=1-2, to=2-2]
		\arrow[""{name=2, anchor=center, inner sep=0}, "{\F}", shift left=3, from=1-1, to=1-2]
		\arrow[""{name=3, anchor=center, inner sep=0}, "{\G}"', shift right=3, from=1-1, to=1-2]
		\arrow["{\,\Sq(\mu)}", shift right=4, shorten <=2pt, shorten >=2pt, Rightarrow, from=0, to=1]
		\arrow["{\,\boldsymbol{\mu}}", shift right, shorten <=2pt, shorten >=2pt, Rightarrow, from=2, to=3]
	\end{tikzcd}\]
\end{Defn}

Note that a composition functor such as $u_! : \C/A \to \C/C$ in fact lifts to a double functor $\u_! : \DJ/A \to \DJ/C$, as the assignment $(\f, a) \mapsto (\f, u.a)$ (trivially) preserves vertical composition. Using the Frobenius structure of $(\DJ, \DK)$ and the adjunction~(\ref{eq:dcatorthogonalcob}) we thus obtains lifts $\u_!\f^*$ and $\g^*\v_!$ as on the left below: 
\begin{equation}\label{eq:dcatalpha}
	\begin{tikzcd}[ampersand replacement=\&, sep=3em]
		{\DJ/B} \& {\DJ/C} \& {\DK/C} \& {\DK/B} \\
		{\Sq(\C/C)} \& {\Sq(\C/B) \rlap{ ,}} \& {\Sq(\C/C)} \& {\Sq(\C/B) \rlap{ .}}
		\arrow[from=1-1, to=2-1]
		\arrow[from=1-2, to=2-2]
		\arrow[""{name=0, anchor=center, inner sep=0}, "{\g^*\v_!}"', shift right=3, from=1-1, to=1-2]
		\arrow[""{name=1, anchor=center, inner sep=0}, "{\u_!\f^*}", shift left=3, from=1-1, to=1-2]
		\arrow[""{name=2, anchor=center, inner sep=0}, "{\Sq(u_!f^*)}", shift left=3, from=2-1, to=2-2]
		\arrow[""{name=3, anchor=center, inner sep=0}, "{\Sq(g^*v_!)}"', shift right=3, from=2-1, to=2-2]
		\arrow[""{name=4, anchor=center, inner sep=0}, "{\v^*\g_*}", shift left=3, from=1-3, to=1-4]
		\arrow[""{name=5, anchor=center, inner sep=0}, "{\f_*\u^*}"', shift right=3, from=1-3, to=1-4]
		\arrow[""{name=6, anchor=center, inner sep=0}, "{\Sq(v^*g_*)}", shift left=3, from=2-3, to=2-4]
		\arrow[""{name=7, anchor=center, inner sep=0}, "{\Sq(f_*u^*)}"', shift right=3, from=2-3, to=2-4]
		\arrow[from=1-3, to=2-3]
		\arrow[from=1-4, to=2-4]
		\arrow["{\,\Sq(\beta)}", shift right=4, shorten <=2pt, shorten >=2pt, Rightarrow, from=6, to=7]
		\arrow["{\,\bb}", shift right, shorten <=2pt, shorten >=2pt, Rightarrow, from=4, to=5]
		\arrow["{\,\Sq(\alpha)}", shift right=4, shorten <=2pt, shorten >=2pt, Rightarrow, from=2, to=3]
		\arrow["{\,\ba}", shift right, shorten <=2pt, shorten >=2pt, Rightarrow, from=1, to=0]
	\end{tikzcd}
\end{equation}
	As $(\DJ, \DK)$ is part of a closed lifting structure, the lifts $\u_!\f^*$ and $\g^*\v_!$ can be transposed to lifts $\v^*\g_*$ and $\f_*\u^*$ using (\ref{eq:dcatorthogonalcob}). Now we can phrase the Beck--Chevalley condition for double categories completely analogously to that for categories. 
\begin{Defn}\label{def:dcatBC}
	A Frobenius structure for a pair $(\DJ, \DK)$ of double categories satisfies the \emph{Beck--Chevalley condition} when for every $\mathbf{uv} : \f \to \g$ in $\DK$ overlying a pullback square $(u, v)$, the Beck--Chevalley isomorphism $\beta$ lifts to a transformation $\bb: \v^*\g_* \To \f_*\u^*$. 
\end{Defn}
Of course, the definitions of Frobenius structure and Beck--Chevalley condition directly apply to closed lifting structures underlying \awfs.

As before, lifting $\beta$ in this way can be done by lifting $\alpha$, which is easier in practice. 
\begin{Prop}\label{prop:dcatbc}
	For a closed lifting structure $(\DJ, \phi, \DK)$ with a Frobenius structure, the Beck--Chevalley isomorphism $\beta$ lifts in the sense of~(\ref{eq:dcatalpha}) iff its mate $\alpha$ does. 
\end{Prop}
\begin{proof}
	By the fact that (\ref{eq:dcatorthogonalcob}) is a 2-adjunction~\cite[Proposition 21]{bourke2016AlgebraicWeak}, in combination with the isomorphisms $\DJ/C \cong \ldpf(\DK/C)$ and $\DK/B \cong (\DJ/B)\ldpf$. 
\end{proof}

In fact, as we will now show, for concrete double functors Definitions~\ref{def:catnatlift} and~\ref{def:dcatnatlift} coincide. This means that checking whether a Frobenius structure for a closed structure $(\DJ, \phi, \DK)$ satisfies the Beck--Chevalley condition comes down to checking whether its Frobenius structure for $(\DJ_1, \DK_1)$, obtained by an application of $(-)_1 : \Dbl \to \Cat$, satisfies the Beck--Chevalley condition. 
\begin{Lemma}\label{lem:concretelift}
	Let $\DJ \to \Sq(\C)$ and $\DK \to \Sq(\D)$ be concrete double functors, and $\F, \G : \DJ \to \DK$ lifts of $F, G: \C \to \D$; then lifts of $\mu : F \to G$ to $\F \To \G$ are in bijection with its lifts to $\F_1 \To \G_1$: 
\[\begin{tikzcd}[ampersand replacement=\&, sep=3em]
	\DJ \& \DK \& {\DJ_1} \& {\DK_1} \\
	{\Sq(\C)} \& {\Sq(\D) \rlap{ ,}} \& {\Ar(\C)} \& {\Ar(\D) \rlap{ .}}
	\arrow[from=1-1, to=2-1]
	\arrow[from=1-2, to=2-2]
	\arrow[""{name=0, anchor=center, inner sep=0}, "\G"', shift right=3, from=1-1, to=1-2]
	\arrow[""{name=1, anchor=center, inner sep=0}, "\F", shift left=3, from=1-1, to=1-2]
	\arrow[""{name=2, anchor=center, inner sep=0}, "{\Sq(F)}", shift left=3, from=2-1, to=2-2]
	\arrow[""{name=3, anchor=center, inner sep=0}, "{\Sq(G)}"', shift right=3, from=2-1, to=2-2]
	\arrow[""{name=4, anchor=center, inner sep=0}, "{\F_1}", shift left=3, from=1-3, to=1-4]
	\arrow[""{name=5, anchor=center, inner sep=0}, "{\G_1}"', shift right=3, from=1-3, to=1-4]
	\arrow[""{name=6, anchor=center, inner sep=0}, "{\Ar(F)}", shift left=3, from=2-3, to=2-4]
	\arrow[""{name=7, anchor=center, inner sep=0}, "{\Ar(G)}"', shift right=3, from=2-3, to=2-4]
	\arrow[from=1-3, to=2-3]
	\arrow[from=1-4, to=2-4]
	\arrow["{\,\Ar(\mu)}", shift right=4, shorten <=2pt, shorten >=2pt, Rightarrow, from=6, to=7]
	\arrow["{\,\bmu}", shift right, shorten <=2pt, shorten >=2pt, Rightarrow, from=4, to=5]
	\arrow["{\,\Sq(\mu)}", shift right=4, shorten <=2pt, shorten >=2pt, Rightarrow, from=2, to=3]
	\arrow["{\,\bmu}", shift right, shorten <=2pt, shorten >=2pt, Rightarrow, from=1, to=0]
\end{tikzcd}\]
\end{Lemma}
\begin{proof} 
By spelling out the definitions. 
\end{proof} 
\begin{Lemma}\label{lem:concreteslice}
	For a closed lifting structure $(\DJ, \phi, \DK)$ and $A \in \C$ the double functors $\DJ/A$ and $\DK/A$ are concrete. 
\end{Lemma}
\begin{proof}
	The double functors $\DJ\ldpf$ and $\ldpf\DJ$ are always concrete for any $\DJ$, and so because the lifting structure is closed so are $\DJ$ and $\DK$. Furthermore, slicing preserves concreteness. 
\end{proof}
\begin{Cor}\label{cor:bcreduction}
	A Frobenius structure for a closed lifting structure $(\DJ, \phi, \DK)$ satisfies the Beck--Chevalley condition when its associated Frobenius structure for $(\DJ_1, \DK_1)$ does. 
\end{Cor}
\begin{proof}
	By Lemmas~\ref{lem:concretelift}~and~\ref{lem:concreteslice}. 
\end{proof}

Lastly, we show that our running example satisfies this condition. 

\begin{Prop}\label{prop:bcfrobstruc}
	The Frobenius structure for $(\DSplRef, \DSplOpFib)$ of Proposition~\ref{prop:dfrobstruc} satisfies the Beck--Chevalley condition. 
\end{Prop}
\begin{proof}
	Consider a pullback square $(X, Y)$ between split opfibrations $P$ and $Q$ which commutes with their splittings. By Corollary~\ref{cor:bcreduction} it suffices to show that the induced transformation $\alpha : X_!P^* \To Q^*Y_!$ has a lift $\ba$ as depicted on the left below: 
\[\begin{tikzcd}[sep=large]
	{\SplRef/\B} & {\SplRef/\C} \\
	{\Ar(\Cat/\B)} & {\Ar(\Cat/\C)} \rlap{ ,}
	\arrow[from=1-1, to=2-1]
	\arrow[from=1-2, to=2-2]
	\arrow[""{name=0, anchor=center, inner sep=0}, "{\X_!\P^*}", shift left=3, from=1-1, to=1-2]
	\arrow[""{name=1, anchor=center, inner sep=0}, "{\Q^*\Y_!}"', shift right=3, from=1-1, to=1-2]
	\arrow[""{name=2, anchor=center, inner sep=0}, "{\Ar(X_!P^*)}", shift left=3, from=2-1, to=2-2]
	\arrow[""{name=3, anchor=center, inner sep=0}, "{\Ar(Q^*Y_!)}"', shift right=3, from=2-1, to=2-2]
	\arrow["{\,\ba}", shift right, shorten <=2pt, shorten >=2pt, Rightarrow, from=0, to=1]
	\arrow["{\,\Ar(\alpha)}", shift right=4, shorten <=2pt, shorten >=2pt, Rightarrow, from=2, to=3]
\end{tikzcd}
\quad\quad
\begin{tikzcd}
	{\mathcal{F}} & {\A\times_\B\mathcal{F}} & {\C\times_\D\mathcal{F}} \\
	\E & {\A \times_\B\E} & {\C\times_\D\E} \\
	\B & \A & \C \\
	& \D \rlap{ .}
	\arrow["P"', from=3-2, to=3-1]
	\arrow["X", from=3-2, to=3-3]
	\arrow["Q", from=3-3, to=4-2]
	\arrow["Y"', from=3-1, to=4-2]
	\arrow["U"', from=2-1, to=3-1]
	\arrow["L", shift left, from=2-1, to=1-1]
	\arrow["R", shift left, from=1-1, to=2-1]
	\arrow["F"', shift right, from=2-2, to=3-2]
	\arrow["{P^*U}", shift left, from=2-2, to=3-2]
	\arrow["G", shift left, from=2-2, to=1-2]
	\arrow["{P^*R}", shift left, from=1-2, to=2-2]
	\arrow["{Q^*R}", shift left, from=1-3, to=2-3]
	\arrow["{G'}", shift left, from=2-3, to=1-3]
	\arrow["{Q^*(Y.U)}", shift left, from=2-3, to=3-3]
	\arrow["{F'}"', shift right, from=2-3, to=3-3]
	\arrow["{\alpha_U}", dotted, from=2-2, to=2-3]
	\arrow["{\alpha_{U.R}}", shift right, dotted, from=1-2, to=1-3]
	\arrow["{\pi_\mathcal{F}}"', from=1-2, to=1-1]
	\arrow["{\pi_\E}"', from=2-2, to=2-1]
\end{tikzcd}
\]
This comes down to showing that the square $(\alpha_{U.R}, \alpha_U) = (X.\pi_\A\times_\D\pi_\mathcal{F}, X.\pi_\A \times_\D\pi_\E)$ is a morphism of split reflections, meaning $(X.\pi_\A\times_\D\pi_\mathcal{F}).G = G'.(X.\pi_\A \times_\D\pi_\E)$, as on the right above. Spelling out the definitions involved we see that this comes down to showing that for any pair $(a, e) \in \A \times_\B \E$ we have $X((U\eta_e)_!a) = (YU\eta_e)_!Xa$. Since $(X, Y)$ is a morphism of opfibrations we have $\smash{\underline{YU\eta_e}} = X\smash{\underline{U\eta_e}}$, from which the desired equality follows immediately. 
\end{proof}

\section{Strong Frobenius structures} We turn to a minor digression concerning strengthened versions of the Frobenius property sometimes encountered in the literature; see e.g.~\cite[Definition~B.6.2]{larrea2018ModelsDependent} and~\cite[Proposition~5.2]{vandenberg2022EffectiveTrivial}. Recall that Definition~\ref{def:wfsfrob} of the Frobenius property for a \wfs\ $(L, R)$ states that $L$ is closed under pullback along $R$; if, in addition, $L$ comes with a notion of structure preserving squares of $L$ maps, then the Frobenius property can be strengthened to demand that the pullback squares of $L$ maps along $R$ maps are structure preserving. We argue that this strengthened Frobenius property can be understood in terms of the notion of natural transformation lifting of Definition~\ref{def:catnatlift}. 

Let $(\J, \K)$ be a pair of functors over an arrow category $\Ar(\C)$, and $f : A \to B$ a map in $\C$. A Frobenius structure for $f$ is a lift $\f^*: \J/B \to \J/A$. A pair $(\g, b)$ in $\J/B$ now gives rise to a pullback square as depicted below, where $\varepsilon$ is the counit of $f_! \ladj f^*$:
\[\begin{tikzcd}[ampersand replacement=\&]
	\bullet \& \bullet \\
	\bullet \& \bullet \\
	A \& {B \rlap{ .}}
	\arrow["{f^*b}"', from=2-1, to=3-1]
	\arrow["b", from=2-2, to=3-2]
	\arrow["f"', from=3-1, to=3-2]
	\arrow["{\varepsilon_b}", from=2-1, to=2-2]
	\arrow["g", from=1-2, to=2-2]
	\arrow["{f^*g}"', from=1-1, to=2-1]
	\arrow["{\varepsilon_{b.g}}", from=1-1, to=1-2]
	\arrow["\lrcorner"{anchor=center, pos=0.125}, draw=none, from=2-1, to=3-2]
	\arrow["\lrcorner"{anchor=center, pos=0.125}, draw=none, from=1-1, to=2-2]
\end{tikzcd}\]
As mentioned previously, for any such $f$ there is a lift $\f_! : \J/A \to \J/B : (\g, a) \mapsto (\g, f.a)$, and so the demand that $(\varepsilon_{b.g}, \varepsilon_b)$ is in the image of a map $\f^*\g \to \g$ in $\J$ can be phrased as the requirement that $\varepsilon$ has a lift $\be : \f_!\f^* \to 1$ in the sense of Definition~\ref{def:catnatlift}. What is called the strong Frobenius property in the literature can then be seen as the special case of this where $b = 1_B$ and $f^*b = 1_A$. 

Thinking of the strong Frobenius property in this way shows that it has an evident counterpart, which is the requirement that the unit $\eta$ of $f_! \ladj f^*$ lifts to $\bn: 1 \to \f^*\f_!$. Together, these lifts could then be further required to satisfy the triangle identities, so that $\f_! \ladj \f^*$. If $(\J, \K)$ is part of a closed lifting structure then we can transpose $\f^*$ and $\f_!$ to obtain a pushforward structure $\f_* : \K/A \to \K/B$ and (by abuse of notation) $\f^* : \K/B \to \K/A$. These can then be subjected to similar conditions, asking for a lifted adjunction $\f^* \ladj \f_*$. For the sake of discussion we will call such lifted adjunctions strong Frobenius and strong pushforward structures. 

\begin{Defn}
	A \emph{strong Frobenius structure} $(\f^*, \bn, \be)$ on a morphism $f : A \to B$ with respect to a closed pair $(\J, \K)$ is a Frobenius structure on $f$ togther with lifts $\bn$ and $\be$ of the unit and counit of $f_! \ladj f^*$, giving rise to an adjunction $\f_! \ladj \f^*$ as on the left of: 
\[\begin{tikzcd}[ampersand replacement=\&, sep=3em]
	{\J/B} \& {\J/A} \& {\K/A} \& {\K/B} \\
	{\Ar(\C/B)} \& {\Ar(\C/A) \rlap{ ,}} \& {\Ar(\C/A)} \& {\Ar(\C/B) \rlap{ .}}
	\arrow[""{name=0, anchor=center, inner sep=0}, "{\Ar(f^*)}", shift left=2, from=2-1, to=2-2]
	\arrow[from=1-1, to=2-1]
	\arrow[from=1-2, to=2-2]
	\arrow[""{name=1, anchor=center, inner sep=0}, "{\f^*}", shift left=2, from=1-1, to=1-2]
	\arrow[""{name=2, anchor=center, inner sep=0}, "{\Ar(f_*)}", shift left=2, from=2-3, to=2-4]
	\arrow[from=1-3, to=2-3]
	\arrow[from=1-4, to=2-4]
	\arrow[""{name=3, anchor=center, inner sep=0}, "{\f_*}", shift left=2, from=1-3, to=1-4]
	\arrow[""{name=4, anchor=center, inner sep=0}, "{\f_!}", shift left=2, from=1-2, to=1-1]
	\arrow[""{name=5, anchor=center, inner sep=0}, "{\Ar(f_!)}", shift left=2, from=2-2, to=2-1]
	\arrow[""{name=6, anchor=center, inner sep=0}, "{\f^*}", shift left=2, from=1-4, to=1-3]
	\arrow[""{name=7, anchor=center, inner sep=0}, "{\Ar(f^*)}", shift left=2, from=2-4, to=2-3]
	\arrow["\dashv"{anchor=center, rotate=90}, draw=none, from=4, to=1]
	\arrow["\dashv"{anchor=center, rotate=90}, draw=none, from=5, to=0]
	\arrow["\dashv"{anchor=center, rotate=90}, draw=none, from=6, to=3]
	\arrow["\dashv"{anchor=center, rotate=90}, draw=none, from=7, to=2]
\end{tikzcd}\]
Similarly, a \emph{strong pushforward structure} $(\f_*, \bar{\bn}, \bar{\be})$ on $f$ is a pushforward structure on $f$ with lifts $\bar{\bn}$ and $\bar{\be}$ of the unit and counit of $f^* \ladj f_*$, such that $\f^* \ladj \f_*$. 
\end{Defn}

\begin{Prop}
	For a map $f: A \to B$ in \ct{C}, and a closed lifting structure $(\J, \phi, \K)$, there is a bijection between strong Frobenius and pushforward structures on $f$.
\end{Prop}
\begin{proof}
	The units and counits of $f_! \ladj f^* \ladj f_*$ are mates~\cite[Lemma~3.1]{hazratpour20242categoricalProof}, and so---due to how the various lifts of $f_!$, $f^*$, and $f_*$ are constructed, and that $(\J, \phi, \K)$ is closed---their lifts are in bijective correspondence by \cite[Proposition 5.8]{gambino2017FrobeniusCondition}. The triangle equalities do not add any extra requirement because $\J$ and $\K$ are faithful (as the lifting structure is closed).  
\end{proof}

Of course, similar definitions and results exist for the double categorical case, and again our running example exhibits such a structure. 

\begin{Prop}
	The pair $(\mathbf{SplRefl}, \SplOpFib)$ admits a strong Frobenius structure.
\end{Prop}
\begin{proof}
	We expand on the proof of Proposition~\ref{prop:functor:frobstruc} by showing that the Frobenius structure $\P^*$ of a split opfibration $P : \A \to \B$ comes with the required lifts. Since $\mathbf{SplRefl}$ is faithful, these lifts are not additional structure but just a property of $\P^*$. To show that $\varepsilon : P_!P^* \to 1$ lifts, we consider the situation as drawn on the left below: 
\[\begin{tikzcd}[ampersand replacement=\&]
	{\A\times\B_\D} \& \D \&[.5cm] \D \& {\A\times_\B\D} \& \D \\
	{\A\times_\B\C} \& \C \& \C \& {\A\times_\B\C} \& \C \\
	\A \& {\B\rlap{ ,}} \& \A \&\& {\B \rlap{ .}}
	\arrow["R", shift left, from=1-3, to=2-3]
	\arrow["L", shift left, from=2-3, to=1-3]
	\arrow["G", shift left, from=2-4, to=1-4]
	\arrow["{P^*R}", shift left, from=1-4, to=2-4]
	\arrow["U"', from=2-3, to=3-3]
	\arrow["{P^*P_!U}"{pos=0.3}, shift left, from=2-4, to=3-3]
	\arrow["F"', shift right, from=2-4, to=3-3]
	\arrow["P"', from=3-3, to=3-5]
	\arrow["R", shift left, from=1-5, to=2-5]
	\arrow["L", shift left, from=2-5, to=1-5]
	\arrow[from=1-4, to=1-5]
	\arrow[from=2-4, to=2-5]
	\arrow["{P_!U}", from=2-5, to=3-5]
	\arrow["{\eta_{U.R}}", from=1-3, to=1-4]
	\arrow["{\eta_U}", from=2-3, to=2-4]
	\arrow["U", from=2-2, to=3-2]
	\arrow["R", shift left, from=1-2, to=2-2]
	\arrow["L", shift left, from=2-2, to=1-2]
	\arrow["F"', shift right, from=2-1, to=3-1]
	\arrow["P"', from=3-1, to=3-2]
	\arrow["{P^*U}", shift left, from=2-1, to=3-1]
	\arrow["{P^*R}", shift left, from=1-1, to=2-1]
	\arrow["G", shift left, from=2-1, to=1-1]
	\arrow["{\varepsilon_{U.R}}", from=1-1, to=1-2]
	\arrow["{\varepsilon_U}", from=2-1, to=2-2]
\end{tikzcd}\]
	Indeed, $(\varepsilon_{U.R}, \varepsilon_U) = (\pi_D, \pi_\C)$ is a morphism of split reflections because for $(a, c) \in \A \times_\B \C$ we have $\pi_\D(G(a, c)) = \pi_\D(F(a, c), Lc) = Lc = L(\pi_\C(a, c))$. For $\eta: 1 \to P^*P_!$ we should show that the square $(\eta_{U.R}, \eta_U) = (U.R \times_\B 1_\D, U \times_{\B} 1_\C)$, drawn on the right above, is a morphism of split reflections, too. For $c \in \C$, the equation $G(\eta_U(c)) = \eta_{U.R}(Lc)$ boils down to the equality $URLc = (PU\eta_c)!Uc$. This is an instance of the general fact that for any $f: a \to b$ in $\A$, the cocartesian lift $\underline{Pf}$ is parallel to $f$ because $P$ is split. Lastly, these lifts satisfy the triangle identities as $\mathbf{SplRefl}$ is faithful. 
\end{proof}

\section{An algebraic model of type theory from groupoids}

Comprehension categories are one of the categorical structures used to model various forms of type theory. Substitution is modeled using pullbacks, and one difficulty with this is that pullbacks are usually only associative up to isomorphism, wheras substitution in type theory is strictly associative. So, in order to faithfully model type theory, the comprehension category in question should be split. One way to do this is to use the right adjoint of the forgetful functor $\SplFib \to \Fib$. In~\cite[Theorem~2.6]{gambino2023ModelsMartinLof}, Gambino and Larrea give a coherence theorem for this method, by identifying pseudo-stability conditions on a comprehension category that ensure that its splitting has strict interpretations of the $\Sigma$-, $\Pi$, and $\mathrm{Id}$-types of Martin-L\"of type theory. Any \awfs\ $(\DL, \DR)$ induces a comprehension category $\DL_1^{\pf} \to \Ar(\C)$, and conditions are identified in~\cite[Theorem~4.11]{gambino2023ModelsMartinLof} on $(\DL, \DR)$ that ensure that the coherence theorem can be applied to it. They then revisit the Hofmann and Streicher groupoid model of~\cite{hofmann1998GroupoidInterpretation}, by exhibiting an \awfs\ on the category $\Gpd$ of groupoids which satisfies these conditions~\cite[Theorem~5.5]{gambino2023ModelsMartinLof}. 

An \awfs\ $(\DL, \DR)$ induces a second (closely related) comprehension category, namely $\DR_1 \to \Ar(\C)$, and our goal in this section is to reproduce the aforementioned results in~\cite{gambino2023ModelsMartinLof} for this second option---that is, to identify conditions on the \awfs\ $(\DL, \DR)$ that ensure that the coherence theorem can be applied to $\DR_1 \to \Ar(\C)$. (Note that there is a functor $\DR_1 \cong (\DL\ldpf)_1 \hookrightarrow \DL_1^{\pf}$ over $\Ar(\C)$, which is generally not invertible; $\DL_1^{\pf}$ corresponds to the retract closure of $\DR_1$.) We phrase these conditions, and then show that the same \awfs\ used by Gambino and Larrea on $\Gpd$ satisfies our conditions. The majority of the work needed for this has been done in the preceeding sections, and what remains is a straightforward adaptation of the approach of~\cite{gambino2023ModelsMartinLof}. Therefore, we only outline the proofs, and the reader is referred to~\cite{gambino2023ModelsMartinLof} for the rest of the details. 

\subsection{Coherence for the comprehension category of right maps} As mentioned, the forgetful functor of the right class of an  \awfs\ is always a comprehension category.

\begin{Prop}
	For an \awfs\ $(\DL, \DR)$, $\DR_1 \to \Ar(\C)$ is a comprehension category.
\end{Prop}
\begin{proof}
	This follows from Proposition~\ref{prop:doublecat:comprehensiondouble} above.
\end{proof}

In fact, no further requirements on the \awfs\ are needed to choose the $\Sigma$-types. 

\begin{Prop}\label{prop:stypes}
	The comprehension category $\DR_1 \to \Ar(\C)$ associated with an \awfs\ admits a pseudo-stable choice of $\Sigma$-types. 
\end{Prop}
\begin{proof}
	The choice of $\Sigma$-types is given simply by the composition operation of $\DR$, which is functorial and preserves pullback squares. The rest of the choices of structure are made as in~\cite[Proposition~4.3]{gambino2023ModelsMartinLof} (and see~\cite[Lemma 2.7.7]{larrea2018ModelsDependent}). 
\end{proof}

The $\Pi$-types are interpreted using a pushforward structure on the \awfs, and the Beck--Chevalley condition ensures this choice is pseudo-stable. Following terminology of~\cite{joyal2017NotesClans} we introduce a shorthand for an \awfs\ with this required structure.

\begin{Defn}
	A $\pi$-\awfs is an \awfs\ with a Frobenius structure satisfying the Beck--Chevalley condition as in Definition~\ref{def:dcatBC}. 
\end{Defn}

\begin{Prop}\label{prop:ptypes}
	The comprehension category $\DR_1 \to \Ar(\C)$ associated with a $\pi$-\awfs\ admits a pseudo-stable choice of $\Pi$-types. 
\end{Prop}
\begin{proof}
	Consider composable $\DR$ maps $\g$ and $\f$. As the \awfs\ has a Frobenius structure, $\f$ has a pushforward structure $\f_*$ by Theorem~\ref{thm:dffequiv}, which gives an $\DR$ map $\f_*\g$ by the reasoning in Proposition~\ref{prop:sc:arrtoobj}; this is the choice of $\Pi$-type for $\g$ and $\f$. The Beck--Chevalley condition ensures that the assignment $\Pi : (\g, \f) \mapsto \f_*\g$ is functorial:
\[\begin{tikzcd}[ampersand replacement=\&,column sep=small,row sep=tiny]
	\bullet \&\& \bullet \\
	\&\&\&\& \bullet \&\& \bullet \&\& \bullet \&\& \bullet \\
	\bullet \&\& \bullet \& \smash{\overset{\Pi}{\longmapsto}} \\
	\&\&\&\& \bullet \&\& \bullet \&\& \bullet \&\& {\bullet \rlap{ ,}} \\
	\bullet \&\& \bullet
	\arrow["u", from=1-1, to=1-3]
	\arrow["\g"', from=1-1, to=3-1]
	\arrow["\lrcorner"{anchor=center, pos=0.125}, draw=none, from=1-1, to=3-3]
	\arrow["\i", from=1-3, to=3-3]
	\arrow["{f_*\alpha}", from=2-5, to=2-7]
	\arrow["{\f_*\g}"', from=2-5, to=4-5]
	\arrow["{\beta_i}", from=2-7, to=2-9]
	\arrow["{\f_*\v^*\i}"', from=2-7, to=4-7]
	\arrow["{w^+}", from=2-9, to=2-11]
	\arrow["{\w^*\h_*\i}", from=2-9, to=4-9]
	\arrow["{\h_*\i}", from=2-11, to=4-11]
	\arrow["v", from=3-1, to=3-3]
	\arrow["\f"', from=3-1, to=5-1]
	\arrow["\lrcorner"{anchor=center, pos=0.125}, draw=none, from=3-1, to=5-3]
	\arrow["\h", from=3-3, to=5-3]
	\arrow["1"', from=4-5, to=4-7]
	\arrow["1"', from=4-7, to=4-9]
	\arrow["w"', from=4-9, to=4-11]
	\arrow["w"', from=5-1, to=5-3]
\end{tikzcd}\]
the outer rectangle on the right above is a pullback, and is a square of $\DR$ because all three of the squares comprising it are. The rest of the choices of structure are made as in~\cite[Proposition~4.6]{gambino2023ModelsMartinLof} (and see~\cite[Lemma 2.7.8]{larrea2018ModelsDependent}). 
\end{proof}

The idea for interpreting identity types using \wfs\ originates from~\cite{awodey2009HomotopyTheoretic} and was further developed in~\cite{vandenberg2012TopologicalSimplicial} by the introduction of the notion of stable functorial choice of path objects. We adapt this definition to the right maps of an \awfs\ as follows. 

\begin{Defn}
	A \emph{functorial factorisation of the diagonal} on a category $\C$ is a functor $P = (r, \rho) : \Ar(\C) \to \Ar(\C) \times_{\C} \Ar(\C)$ such that $\rho f.rf = \delta_f$ for any map $f : A \to B$ in $\C$, where $\delta_f : A \to A \times_B A$ is the diagonal morphism. Such a factorisation is called \emph{stable} if its right leg $\rho$ preserves pullback squares. 
\end{Defn}
\begin{Defn}
	A \emph{stable functorial choice of path objects} (\sfpo) on an \awfs\ $(\DL, \DR)$ is a lift $\P$ of a stable functorial factorisation of the diagonal $P$:
	\[
	\begin{tikzcd}
		\DR_1 \ar[r, "\P"] \ar[d] & \DL_1 \times_\C \DR_1 \ar[d] \\
		\Ar(\C) \ar[r, "P"] & \Ar(\C) \times_\C \Ar(\C) \rlap{ .}
	\end{tikzcd}
	\] 
\end{Defn}

\begin{Prop}\label{prop:idtypes}
	The comprehension category $\DR_1 \to \Ar(\C)$ associated with an \awfs\ with an \textsc{sfpo} admits a pseudo-stable choice of $\mathrm{Id}$-types. 
\end{Prop}
\begin{proof}
	The choice of $\mathrm{Id}$-type for an $\DR$ map $\f$ is given by the right leg $\bm{\rho}\f$ of the \sfpo\ $\P = (\bm{r}, \bm{\rho})$; this assignment is functorial and preserves cartesian morphisms by definition. The rest of the choices of structure are made as in~\cite[Proposition~4.9]{gambino2023ModelsMartinLof}.
\end{proof}

We thus obtain the desired analog of \cite[Theorem~4.11]{gambino2023ModelsMartinLof}. 

\begin{Thm}\label{thm:stablechoices}
	A $\pi$-\awfs\ with an \sfpo\ induces a comprehension category with strictly stable choices of $\Sigma$-, $\Pi$-, and $\mathrm{Id}$-types. 
\end{Thm}
\begin{proof}
	By Propositions~\ref{prop:stypes}, \ref{prop:ptypes}, and \ref{prop:idtypes} the comprehension category associated with the \awfs\ has pseudo stable choices of these types, and so its right adjoint splitting has strictly stable choices by~\cite[Theorem~2.6]{gambino2023ModelsMartinLof}.
\end{proof}

\subsection{The groupoid model}

Recall from Example~\ref{ex:lawfs} that the double categories $(\DSplRef(\Cat), \DSplFib(\Cat))$ on $\Cat$ of split reflections and split fibrations of categories form an \awfs. This pair does not readily admit a Frobenius structure satisfying the Beck--Chevalley condition---for that we need the right class to consist of split \emph{op}fibrations, as shown in Proposition~\ref{prop:bcfrobstruc}. To remedy this, we lift this \awfs\ on $\Cat$ along the inclusion $\Gpd \to \Cat$ to one on $\Gpd$, because for groupoids the notions of fibration and opfibration coincide. 

\begin{Prop}\label{prop:gpdpawfs}
	The pair of double categories $(\DSplRef(\Gpd), \DSplFib(\Gpd))$ over $\Sq(\Gpd)$, together with the lifting operation of Example~\ref{ex:liftingstrforrari}, is a $\pi$-\awfs on $\Gpd$.
\end{Prop}
\begin{proof}
	This is the projective/injective lift of the \awfs\ $(\DSplRef(\Cat), 
	\DSplFib(\Cat))$ along the inclusion $\Gpd \to \Cat$, see \cite[Section~4.5]{bourke2016AlgebraicWeak}. Fibrations of groupoids are also opfibrations, so there is an isomorphism $\DSplFib(\Gpd) \cong \DSplOpFib(\Gpd)$. Hence, this \awfs\ gets its Frobenius structure from Propositions~\ref{prop:dfrobstruc}, which satisfies the Beck--Chevalley condition as per Proposition~\ref{prop:bcfrobstruc}. 
\end{proof}

\begin{Prop}\label{prop:gpdawfssfpo}
	The \awfs\ $(\DSplRef(\Gpd), \DSplFib(\Gpd))$ admits an \sfpo. 
\end{Prop}
\begin{proof}
	A straightforward adaptation of~\cite[Proposition~3.5]{gambino2023ModelsMartinLof}.
\end{proof}

This gives us the following analog of~\cite[Theorem~5.5]{gambino2023ModelsMartinLof}.

\begin{Thm}
	The right adjoint splitting of $\SplRef(\Gpd) \to \Ar(\Gpd)$ yields a model of dependent type theory with $\Sigma$-, $\Pi$-, and $\mathrm{Id}$-types. 
\end{Thm}
\begin{proof}
By Propositions~\ref{prop:gpdpawfs}, \ref{prop:gpdawfssfpo}, and Theorem~\ref{thm:stablechoices}.
\end{proof}
\section{Conclusion}

In this paper we have obtained a Frobenius equivalence and formulated a suitable version of the Beck--Chevalley condition for algebraic weak factorisation systems. We have shown how these notions can be used to obtain models of type theory with dependent function types, as illustrated by split fibrations of groupoids, the basis for the groupoid model of Hofmann and Streicher \cite{hofmann1998GroupoidInterpretation}.

In this way our work is similar to~\cite{larrea2018ModelsDependent,gambino2023ModelsMartinLof}, where the authors also utilize an analogical statement of the Frobenius equivalence for {\awfs}s, which is stated and proven in~\cite{gambino2017FrobeniusCondition}. However, their version differs from the one we propose here in several ways. First of all, they develop a version of the Frobenius equivalence for {\awfs}s which are cofibrantly generated by a category; we make no such assumption. In addition, their works take the ``object view'' rather than the ``arrow view'', in that they primarily view pullback and pushforward as operations on objects rather than arrows of slice categories. We hope to have demonstrated here that it is conceptually simpler to take the arrow view. Finally, the works of~\cite{gambino2017FrobeniusCondition, larrea2018ModelsDependent,gambino2023ModelsMartinLof} use an interpretation of type theory in which the dependent types are interpreted using the algebras of the pointed endofunctor associated with an \awfs, rather than the algebras of the monad associated with an \awfs. In that way their framework targets a different class of examples and does not cover examples like split fibrations of groupoids.

In future work we hope to find more examples. In particular, we would like to show how our ideas apply to the \emph{effective Kan fibrations} from \cite{vandenberg2022EffectiveTrivial}. These are intended to be a good constructive analogue of the Kan fibrations from simplicial homotopy theory and to lead to a constructive account of Voevodsky's model of homotopy type theory in simplicial sets \cite{kapulkin2021SimplicialModel}. In \cite{vandenberg2024ExamplesCofibrant} the second author has shown together with Freek Geerligs how the effective Kan fibrations appear as the right class in an \awfs, so our current framework is the appropriate one for these maps.


\printbibliography

\appendix

\end{document}